\documentclass[11pt]{amsart}
\usepackage{amsmath}
\usepackage{amssymb}
\usepackage{amscd}
\usepackage{amsthm}
\usepackage{enumerate}
\usepackage{fullpage}
\usepackage{mathrsfs}
\usepackage{setspace}
\usepackage{mathtools}
\usepackage{xcolor}
\usepackage[all]{xy}
\usepackage{url}
\usepackage[colorlinks=true,linkcolor=blue,breaklinks,pagebackref]{hyperref}
\usepackage[lite,abbrev,alphabetic]{amsrefs}
\usepackage[poorman]{cleveref}
\usepackage{bbm}
\usepackage{float} 

\usepackage{booktabs,longtable}

\usepackage{tikz}
\usetikzlibrary{calc}

\theoremstyle{definition}
\newtheorem{definition}{Definition}[section]

\theoremstyle{plain}
\newtheorem{theorem}[definition]{Theorem}
\newtheorem{proposition}[definition]{Proposition}
\newtheorem{lemma}[definition]{Lemma}
\newtheorem{corollary}[definition]{Corollary}

\newtheorem*{theorem*}{Theorem}

\theoremstyle{remark}
\newtheorem{remark}[definition]{Remark}

\crefname{theorem}{Theorem}{Theorems}
\crefname{proposition}{Proposition}{Propositions}
\crefname{lemma}{Lemma}{Lemmas}
\crefname{corollary}{Corollary}{Corollaries}
\crefname{conjecture}{Conjecture}{Conjectures}
\crefname{hypothesis}{Hypothesis}{Hypotheses}
\crefname{remark}{Remark}{Remarks}
\crefname{condition}{Condition}{Conditions}
\crefname{example}{Example}{Examples}

\def\N{\mathbb{N}}
\def\C{\mathbb{C}}
\def\R{\mathbb{R}}
\def\Q{\mathbb{Q}}
\def\Z{\mathbb{Z}}
\def\A{\mathbb{A}}

\def\GL{\mathrm{GL}}

\def\SL{\mathrm{SL}}
\def\PSL{\mathrm{PSL}}

\def\M{\mathrm{M}}

\def\rd{\,\mathrm{d}}

\def\inf{\infty}

\def\bs{\backslash}

\def\1{\mathbf{1}}
\def\0{\mathbf{0}}

\DeclareMathOperator{\vol}{vol}

\DeclareMathOperator{\tr}{tr}
\DeclareMathOperator{\id}{id}

\newcommand{\fa}{\mathfrak{a}}
\newcommand{\fA}{\mathfrak{A}}

\newcommand{\fD}{\mathfrak{D}}

\newcommand{\ff}{\mathfrak{f}}
\newcommand{\fF}{\mathfrak{F}}

\newcommand{\fh}{\mathfrak{h}}

\newcommand{\fo}{\mathfrak{o}}
\newcommand{\fp}{\mathfrak{p}}

\newcommand{\fL}{\mathfrak{L}}

\newcommand{\fT}{\mathfrak{T}}

\newcommand{\fX}{\mathfrak{X}}

\newcommand{\bT}{\mathbb{T}}

\newcommand{\cD}{\mathcal{D}}

\newcommand{\cH}{\mathcal{H}}
\newcommand{\cI}{\mathcal{I}}

\newcommand{\cL}{\mathcal{L}}

\newcommand{\cP}{\mathcal{P}}

\newcommand{\rG}{\mathrm{G}}

\numberwithin{equation}{section}

\begin{document}

\title{The Eichler--Selberg trace formula for Hilbert cusp forms, the class numbers of quartic CM fields, and their distributions}
\author{Seiji Kuga}
\author{Andrei Seymour-Howell}
\author{Satoshi Wakatsuki}
\address{Faculty of Science and Technology, Sophia University, Kioi-cho 7-1 Chiyoda-ku Tokyo, 102-8554, Japan}
\email{s-kuga-2g7@sophia.ac.jp}
\address{Department of Mathematics Education, Chonnam National University, Gwangju 61186,
Republic of Korea}
\email{aseymourhowell@jnu.ac.kr}
\address{Faculty of Mathematics and Physics, Institute of Science and Engineering, Kanazawa University, Kakumamachi, Kanazawa, Ishikawa, 920-1192, Japan}
\email{wakatsuk@staff.kanazawa-u.ac.jp}

\date{\today}

\begin{abstract}
Motivated by Su's construction of Cohen-type Eisenstein series of half-integral weight over totally real number fields \cite{Su16}, we introduce a generalization of Hurwitz class numbers to totally real number fields.
Using these generalized Hurwitz class numbers, we establish an Eichler--Selberg trace formula for the space of holomorphic Hilbert cusp forms over real quadratic fields of narrow class number one.
While the classical Hurwitz class numbers are defined in terms of class numbers of imaginary quadratic fields, the generalized Hurwitz class numbers appearing in our Eichler--Selberg trace formula are defined in terms of class numbers of quartic CM fields.
For applications of this Eichler--Selberg trace formula, we study the distribution of the generalized Hurwitz class numbers, prove class number relations, and carry out numerical computations of traces of Hecke operators for $\Q(\sqrt{5})$ and $\Q(\sqrt{29})$.
\end{abstract}

\maketitle
\tableofcontents

\section{Introduction and main results}\label{sec:intro}

In this paper, we establish an Eichler--Selberg trace formula for the space of holomorphic Hilbert cusp forms over real quadratic fields of narrow class number one.
For applications, we study the distribution and class number relations of generalized Hurwitz class numbers associated with quartic CM fields, and carry out numerical computations of traces of Hecke operators.
In this section, we first review the classical Hurwitz class numbers and the Eichler--Selberg trace formula.
We then introduce the generalized Hurwitz class numbers, and proceed to the Eichler--Selberg trace formula for spaces of holomorphic Hilbert cusp forms and the distribution of the generalized Hurwitz class numbers.

\subsection{The classical Hurwitz class numbers}
Before stating our main results, we first recall the classical Hurwitz class numbers and the Eichler--Selberg trace formula in the usual setting.
For a natural number $n$, the Hurwitz class number $H(n)$ is defined as the weighted number of $\PSL_2(\Z)$-equivalence classes of positive definite integral binary quadratic forms of discriminant $-n$, where each class is counted with weight equal to the reciprocal of the order of its stabilizer.
If $-n=df^2$ for a fundamental discriminant $d<0$ and a natural number $f$, then for $E=\Q(\sqrt{d})$ we have
\begin{equation}\label{eq:1}
H(n)=h_{E} \, [\fo_E^\times:\Z^\times]^{-1} \, \sum_{u\mid f} u \prod_{p\mid u}(1-(\tfrac{d}{p})\, p^{-1})    
\end{equation}
and $H(n)=0$ otherwise, cf. \cite{Coh75}.
Here $h_E$ denotes the class number of the imaginary quadratic field $E$, $\fo_E$ its ring of integers, $(\tfrac{d}{p})$ the Kronecker symbol, $u$ runs over all positive divisors of $f$, and $p$ runs over all prime divisors of $u$.
In particular, if $-n>4$ is a fundamental discriminant, then $H(n)$ coincides with $h_E/2$.
Moreover, for $n<0$ we set $H(n)=0$, and we define
\begin{equation}\label{eq:2}
H(0)\coloneqq\zeta(-1)=-\frac{1}{12}    
\end{equation}
in terms of the Riemann zeta function $\zeta(s)$.
Let $S_k(\SL_2(\Z))$ denote the space of holomorphic cusp forms of weight $k$ for $\SL_2(\Z)$, and for a natural number $n$, let $T_n$ denote the $n$-th Hecke operator.
Then, for even integers $k\ge 2$, the Eichler--Selberg trace formula
\begin{equation}\label{eq:3}
\tr T_n|_{S_k(\SL_2(\Z))}=-\frac{1}{2}\sum_{t\in\Z}P_k(t,n)\, H(4n-t^2)-\frac{1}{2}\sum_{uu'=n}\min(u,u')^{k-1}+\begin{cases}
 \sum_{a\mid n}a  &  \text{if $k=2$}, \\
   0 & \text{otherwise,}
\end{cases}
\end{equation}
holds; see e.g. \cites{Eic55}, \cite{Sel56}, \cite{Lan76}*{Ch.3, Appendix by D. Zagier}, and \cite{Zag77} (see \cite{Eic57} for weight $2$).
Here we set
\[
P_k(t,n)\coloneqq \frac{\rho^{k-1}-\overline{\rho}^{k-1}}{\rho-\overline{\rho}}, \quad \rho+\overline{\rho}=t, \quad \rho\overline{\rho}=n.
\]
It follows from \eqref{eq:3} that $\dim S_k(\SL_2(\Z))=0$ for $k=2$, $4$, $6$, $8$, $10$, and $14$.
Hence, for these weights, the left-hand side of \eqref{eq:3} is always $0$, and then, as pointed out in \cite{Eic55}*{Added at proof}, the class number relations of Kronecker \cite{Kro60}, Hurwitz \cite{Hur85}, and Eichler \cite{Eic55}*{(10)} follow from \eqref{eq:3}.
Moreover, as shown in \cites{HZ76,Zag75}, the generating function of $H(n)$ is a mock modular form of weight $3/2$.
For recent developments on class number relations and mock modular forms, see \cite{Mer14}.

\subsection{The generalized Hurwitz class numbers}
From the viewpoint of the expression \eqref{eq:1}, one can expect an appropriate generalization of Hurwitz class numbers from the Hilbert analogue of the Cohen Eisenstein series given in \cite{Su16}.
Let $F\ne\Q$ be a totally real number field and denote by $\fo_F$ the integer ring of $F$. For $\kappa\in\Z_{\ge1}$ and a character $\chi$ of the ideal class group of $F$, Su constructed a Hilbert modular form of half-integral parallel weight $(\kappa+\tfrac{1}{2},\cdots,\kappa+\tfrac{1}{2})$ \cite{Su16}*{Theorem 0.1}, which is denoted by $G_{\kappa+\frac{1}{2}}(\cdot,\chi)$ and is a generalization of the Cohen Eisenstein series. For a totally positive element $\xi\in\fo_F$ being congruent to $(-1)^{\kappa}$ times some square in $\fo_F$ modulo $4$, the $\xi$th Fourier coefficient of $G_{\kappa+\frac{1}{2}}(\cdot,\chi)$ is given by
\begin{align*}
C_F^{\kappa,\chi}(\xi):=&\chi(\fD_{(-1)^{\kappa}\xi})L_F(1-\kappa,\overline{\fX_{(-1)^{\kappa}\xi}\chi})\\
&\times
\sum_{\substack{\fa\subset\fo_F:\text{ideal}\\
\fa|\fF_{(-1)^{\kappa}\xi}}}\mu(\fa)\fX_{(-1)^{\kappa}\xi}(\fa)\chi(\fa) \, N_{F/\Q}(\fa)^{\kappa-1} \, \sigma_{2\kappa-1,\chi}(\fF_{(-1)^{\kappa}\xi} \, \fa^{-1}),
\end{align*}
where $\mu$ is the M\"{o}bius function, $\fD_{\xi}$ is the relative discriminant of $F(\sqrt{\xi})/F$, $\fX_{\xi}$ is the quadratic character on the ideal group of $F$ corresponding to $F(\sqrt{\xi})/F$ via the global class field theory, $\fF_{\xi}$ is the integral ideal defined by $(\xi)=\fD_{\xi}\fF_{\xi}^2$, and $\sigma_{k,\chi}$ is the divisor function given as
$$
\sigma_{k,\chi}(\fA)=\sum_{\fa|\fA}\chi(\fa)\, N_{F/\Q}(\fa)^k.
$$
We also write, for an ideal $\fa$ of $\fo_F$, $N_{F/\Q}(\fa)\coloneqq \#(\fo_F/\fa)$ for its norm. Let $\chi_0$ denote the trivial character on the ideal class group of $F$, and 
$h_{E}$ the (wide) class number of the number field $E$.
Since the class number formula implies
\[
L_F(0,\mathfrak{X}_{-\xi})=2^{[F:\Q]-1}[\fo_{F(\sqrt{-\xi})}^\times\colon\fo_F^\times]^{-1}\frac{h_{F(\sqrt{-\xi})}}{h_F},
\]
we have that $C_F^{1,\chi_0}(\xi)$ equals $2^{[F:\Q]-1}H_F(\xi)$, where
\begin{align}
H_F(\xi)\coloneqq\frac{h_{F(\sqrt{-\xi})}}{h_{F}}[\fo_{F(\sqrt{-\xi})}^\times\colon\fo_F^\times]^{-1}\sum_{\fa|\fF_{-\xi}}N_{F/\Q}(\fa)\prod_{\substack{\fp:\text{prime}\\ \fp|\fa}}
\left(
1-(\tfrac{-\xi}{\fp}) N_{F/\Q}(\fp)^{-1}
\right)\label{GHCN}.
\end{align}
Here, the symbol $(\tfrac{-\xi}{\fp})$ takes the value $1$, $-1$, or $0$ according to if the prime ideal $\fp$ is split, inert, or ramified in $F(\sqrt{-\xi})$, that is, it is the Artin symbol.
We shall call this quantity $H_F(n)$ the generalized Hurwitz class number.

Fix a fundamental discriminant $D>0$. 
From this point on in the introduction, we restrict to $F=\Q(\sqrt{D})$, that is, we consider only the real quadratic field $F$ of discriminant $D$.
Here we take $\sqrt{D}$ to be the positive real number, and note that an embedding of $F$ into $\R$ is fixed once and for all. 
Write $\sigma$ for the non-trivial Galois element for $F/\Q$, that is, $\sigma(\sqrt{D})=-\sqrt{D}$. 
We also suppose that the narrow class number of $F$ is $1$.
Under this assumption, since the norm of a fundamental unit is $-1$, it follows from Pell's equation that $D$ has no prime factor congruent to $3$ modulo $4$.
In particular, if $D$ is divisible by $4$, then $D$ has $2$ as a prime factor. 
An element $a \in F$ is said to be totally positive if $a>0$ and $\sigma(a)>0$; 
in this case, we write $a \gg 0$.
Under these assumptions, by \eqref{GHCN}, for $n\in\fo_F$ with $n\gg 0$, we define
\[
H_F(n)\coloneqq h_{F(\sqrt{-n})} \, [\fo_{F(\sqrt{-n})}^\times\colon\fo_F^\times]^{-1} \sum_{\fa|\fF_{-n}}N_{F/\Q}(\fa)\prod_{\fp|\fa}
\left(
1-(\tfrac{-n}{\fp}) N_{F/\Q}(\fp)^{-1}
\right)
\]
when $-n$ is congruent modulo $4$ to some square in $\fo_F$, and set $H_F(n)=0$ otherwise.
If $n\in \fo_F$ is not totally positive, then from the volume of the fundamental domain of the Hilbert modular group (see \cite{Shi63}*{p.70}), and as an analogue of \eqref{eq:2}, we define for $n=0$,
\[
H_F(0)\coloneqq \frac{1}{2}\zeta_F(-1)=\frac{1}{48}B_{2,\chi_D}=\frac{1}{48D}\sum_{a=1}^{D-1} \chi_D(a)\, a^2
\]
and set $H_F(n)=0$ otherwise.
Here, we set $\chi_D(a)\coloneqq \left( \tfrac{D}{a}\right)$, and $B_{k,\chi_D}$ denotes the $k$-th generalized Bernoulli number for $\chi_D$.
The subgroup $U_F$ of totally positive units is generated by $\varepsilon^2$, and we have
\begin{equation}\label{eq:HF1}
H_F(n)=H_F(un)    
\end{equation}
for any $n\in\fo_F$ and any $u\in U_F$.
We note that $F(\sqrt{-n})$ and $F(\sqrt{-\sigma(n)})$ are isomorphic over $\Q$, and hence 
\begin{equation}\label{eq:HF2}
H_F(n)=H_F(\sigma(n))
\end{equation}
holds for any $n\in\fo_F$.
Since there exists $d\in\fo_F$ with $d\gg 0$ such that $\fD_{-n}=(-d)$, that is, $(-n)=(-d)\, \fF_{-n}^2$, by \cite{Oka02}*{Lemma 2.6} we obtain
\begin{equation}\label{eq:HF3}
H_F(n)=H_F(d)\prod_{\fp\mid \fF_{-n}}\frac{N_{F/\Q}(\fp)^{\mathrm{ord}_\fp(\fF_{-n})+1}-N_{F/\Q}(\fp)^{\mathrm{ord}_\fp(\fF_{-n})}(\tfrac{-n}{\fp})-1+(\tfrac{-n}{\fp})}{N_{F/\Q}(\fp)-1},
\end{equation}
where $\fp$ runs over all prime ideals of $\fo_F$ dividing $\fF_{-n}$ and $\mathrm{ord}_\fp(\fa)$ denotes the exponent of $\fp$ in the prime ideal decomposition of $\fa$.

For $n\gg 0$, $F(\sqrt{-n})$ is a CM extension of $F$, that is, a quartic CM field over $\Q$.
As stated in \cite{ST61}*{Example 8.4 (2)}, $F(\sqrt{-n})$ falls into one of the following three cases:
\begin{itemize}
    \item $F(\sqrt{-n})$ is a biquadratic extension of $\Q$.
    \item $F(\sqrt{-n})$ is a cyclic quartic extension of $\Q$.
    \item $F(\sqrt{-n})$ is a non-Galois, and its Galois closure has Galois group $D_4$.
\end{itemize}
When $(-n)=\fD_{-n}$, the number $H_F(n)$ agrees with the class number of $F(\sqrt{-n})$, up to the index of the unit groups. 
In the first case where $F(\sqrt{-n})$ is biquadratic over $\Q$, class field theory implies that
\begin{equation}\label{eq:biquad}
h_{F(\sqrt{-n})} \, [\fo_{F(\sqrt{-n})}^\times\colon\fo_F^\times]^{-1}=\frac{1}{2}\, h_{\Q(\sqrt{-n})}[\fo_{\Q(\sqrt{-n})}^\times\colon\Z^\times]^{-1}\, h_{\Q(\sqrt{-nD})}[\fo_{\Q(\sqrt{-nD})}^\times\colon\Z^\times]^{-1} 
\end{equation}
holds. 
Hence, the essentially new class numbers arise in the other two cases.
These class numbers also admit geometric interpretations; for instance, they are closely related to the number of isomorphism classes of principally polarized abelian surfaces with complex multiplication (see, e.g., \cite{Str14}).
There have already been various studies on class numbers, including classifications of quartic CM fields with class number $1$ (see, e.g., \cite{SO94}), but it appears that there has been no work using the Eichler--Selberg trace formula, which is the main theme of this paper.
We also study the distribution of $H_F(n)$. Since $F(\sqrt{-n})$ is Galois over $\Q$ if and only if $N_{F/\Q}(n)\in(\Z)^2\sqcup D\,(\Z)^2$, it follows that for most $n\gg 0$, $F(\sqrt{-n})$ is non-Galois over $\Q$. Thus, we are led to consider the distribution of class numbers in the non-Galois case.

\subsection{The Eichler--Selberg trace formula for Hilbert cusp forms}
We now describe our Eichler--Selberg trace formula.
Let $S_{\underline{\kappa}}(\SL_2(\fo_F))$ denote the space of Hilbert cusp forms 
of weight $\underline{\kappa}=(\kappa_1,\kappa_2)$ for the congruence subgroup 
$\SL_2(\fo_F)$. 
For a totally positive element $n\in\fo_F$, we denote by $\bT_n$ 
the $n$-th Hecke operator acting on $S_{\underline{\kappa}}(\SL_2(\fo_F))$.
For the definition of $\bT_n$, see \eqref{eq:bTn} 
or \cite{Gar90}*{\S1.15}.
Then, the Hecke operator $T_{n,F}$ analogous to the above $T_n$ can be defined by adjusting the scalar multiple as
\begin{equation}\label{eq:TnFbTn}
T_{n,F}\coloneqq n^{\frac{\kappa_1}{2}-1}\sigma(n)^{\frac{\kappa_2}{2}-1} \bT_n.    
\end{equation}
Note that $\bT_n$ does not depend on multiplying $n$ by an element of $U_F$.
Under these assumptions, an analogue of the Eichler--Selberg trace formula \eqref{eq:3} is given as follows.
\begin{theorem}\label{thm:ES}
If $\kappa_1$ and $\kappa_2$ are even and greater than $1$, then for any totally positive element $n\in\fo_F$, we obtain
\if0\[
\tr\, T_{n,F}|_{S_{\underline{\kappa}}(\SL_2(\fo_F))}  =  \frac{1}{2}\sum_{t\in\fo_F} H_F(4n-t^2)\, P_{\kappa_1}(t,n)\, P_{\kappa_2}(\sigma(t),\sigma(n))   \nonumber  -  \begin{cases}
   \sum_{\fa\mid(n)}N_{F/\Q}(\fa) & \text{if $\kappa_1=\kappa_2=2$,} \\ 0 & \text{otherwise.}
\end{cases}
\]\fi
\begin{align*}
\tr\, T_{n,F}|_{S_{\underline{\kappa}}(\SL_2(\fo_F))}  = &  \frac{1}{2}\sum_{t\in\fo_F} H_F(4n-t^2)\, P_{\kappa_1}(t,n)\, P_{\kappa_2}(\sigma(t),\sigma(n))   \nonumber \\
& -  \begin{cases}
   \sum_{\fa\mid(n)}N_{F/\Q}(\fa) & \text{if $\kappa_1=\kappa_2=2$,} \\ 0 & \text{otherwise.}
\end{cases} \nonumber
\end{align*}
\end{theorem}
\begin{proof}
This formula essentially follows from Saito's explicit formula \cite{Sai84}*{Theorem 2.1}. 
In his paper, the space of adelic automorphic forms on $\GL(2)$ is considered, but by Lemma \ref{lem:32} of this paper, it coincides with the usual space of holomorphic Hilbert cusp forms. 
Although his formula is complicated, one can see from the proof of \cite{Oka02}*{Proposition 2.7} that the present formula follows from it.

In this paper, we give a simple alternative proof without using the explicit formula of \cite{Sai84}. 
More precisely, this theorem follows from Theorem \ref{thm:traceformula}, Proposition \ref{prop:KL2}, and Lemma \ref{lem:KL3} (see also Remark \ref{rem:conjugacyclass}). 
In Theorem \ref{thm:traceformula}, we use the $l$-th Chebyshev polynomial of the second kind $U_l$, but it can be rewritten by means of the relation
\begin{equation}\label{eq:PU}
P_\kappa(t,n)=n^{\frac{\kappa}{2}-1} \, U_{\kappa-2}\left(\frac{t}{2\sqrt{n}}\right).    
\end{equation}
Although the explicit formula of \cite{Sai84} is more general, its proof uses the method of \cite{Hij74}, and therefore involves computations of orbital integrals of elliptic conjugacy classes over individual $p$-adic fields, making the argument rather lengthy. 
In this paper, by using the method of \cite{KL06}, we avoid computations of local orbital integrals over individual $p$-adic fields and give a more transparent alternative proof directly relating generalized Hurwitz class numbers to global orbital integrals of elliptic conjugacy classes (see Propositions \ref{prop:KL} and \ref{prop:KL2}).
\end{proof}
\begin{remark}
\begin{itemize}
    \item As mentioned in the proof, this formula is not essentially new. 
However, it is clear that expressing traces of Hecke operators in such a simple form is important for applications to analytic number theory and computational number theory. 
In particular, since the formula of \cite{Sai84} also contains traces of Atkin--Lehner operators, we plan to rewrite his formula in a simpler form and apply it to the study of root numbers and murmurations in future work.
\item Here we mention some previous works on trace formulas for Hecke operators acting on spaces of holomorphic Hilbert cusp forms. 
Except for the formula for the contribution of elliptic conjugacy classes due to \cite{Hij74} and \cite{Sai84}, a general formula was given in \cite{Shi65}*{\S3}. 
Originally, the assumptions $\kappa_1>2$ and $\kappa_2>2$ were imposed, but these assumptions were removed in \cite{Ish74}.
An analogue of the Eichler--Selberg trace formula for such Hilbert cusp forms is also studied in \cite{Tak86}.
In the formula \cite{Tak86}*{Theorem 2}, the part described in Theorem \ref{thm:ES} by generalized Hurwitz class numbers is expressed instead in terms of the discrete invariants $C_{n,m}$ defined in \cite{Tak86}*{p.144}.
For this reason, it seems that Theorem \ref{thm:ES} cannot be derived directly from \cite{Tak86}*{Theorem 2}.
\item In Theorem \ref{thm:ES}, we assume that the narrow class number of $F$ is $1$, which means that the wide class number of $F$ is $1$ and that the norm of a fundamental unit is $-1$.
To relate the geometric side of the trace formula to generalized Hurwitz class numbers, we use the argument of \cite{KL06}; however, in order to generalize their argument over $\Q$ to $F$, the assumption that the wide class number is $1$ is required, see \S\ref{sec:Elliptic}.
The assumption on units is used in the diffeomorphism \eqref{eq:diff2} of the arithmetic quotient.
This diffeomorphism is used in the interpretation of the spectral side in \cite{Art89}, and is in particular related to the contribution of the trivial representation (cf. the proof of Lemma \ref{lem:31}).
\end{itemize}
\end{remark}

\subsection{Corollaries obtained from the Eichler--Selberg trace formula}

As a corollary of Theorem \ref{thm:ES}, together with \eqref{eq:PU}, we obtain the following dimension formula.
\begin{corollary}\label{cor:main1}
If $\kappa_1$ and $\kappa_2$ are even and greater than $1$, we obtain
\begin{align*}
\dim S_{\underline{\kappa}}(\SL_2(\fo_F))   = &  H_F(0)\, (\kappa_1-1)(\kappa_2-1)  \\
&  + \frac{1}{2}H_F(4)\, (-1)^{\frac{\kappa_1+\kappa_2}{2}}+ H_F(3)\, U_{\kappa_1-2}\left( \frac{1}{2} \right)\, U_{\kappa_2-2}\left( \frac{1}{2} \right)  ,  \nonumber \\
& +I(5)+I(8) -  \begin{cases}
   1  & \text{if $\kappa_1=\kappa_2=2$,} \\ 0 & \text{otherwise.}
\end{cases} \nonumber
\end{align*} 
Here, 
\[
H_F(4)=\frac{1}{4} \times h_{\Q(\sqrt{-D})} \times \begin{cases}
    1 & \text{if $D\equiv 1 \mod 4$},\\
    3 & \text{if $D$ has $2$ as a prime factor},
\end{cases} \qquad 
H_F(3)=\frac{1}{6} \times h_{\Q(\sqrt{-3D})}.
\]
We put
\[
I(5)= \frac{1}{5}\left\{ U_{\kappa_1-2}\left( a_1 \right)\, U_{\kappa_2-2}\left( a_2 \right) + U_{\kappa_1-2}\left( a_2 \right)\, U_{\kappa_2-2}\left( a_1 \right) \right\}, \quad a_j\coloneqq \cos\left(\frac{2j\pi}{5}\right)  \quad \text{if $D=5$},
\]
and $I(5)=0$ if $D\neq 5$.
We also put
\[
I(8)=\frac{1}{4} U_{\kappa_1-2}\left( \frac{1}{\sqrt{2}} \right)\, U_{\kappa_2-2}\left( \frac{1}{\sqrt{2}} \right) \quad \text{if $D=8$},
\]
and $I(8)=0$ if $D\neq 8$.
\end{corollary}
\begin{proof}
If one determines $t\in\fo_F$ satisfying $4-t^2\gg0$ and then computes $H_F(4-t^2)$, the above dimension formula follows from Theorem \ref{thm:ES}.
By direct computation, the elements $t\in\fo_F$ satisfying $4-t^2\gg0$ are given as follows.
\[
t=\begin{cases}
  0,\;\; \pm 1 \;\; \text{or}\;\; \pm\frac{1\pm\sqrt{5}}{2}   & \text{when $D=5$,} \\
  0,\;\; \pm 1 \;\; \text{or}\;\; \pm\sqrt{2}  & \text{when $D=8$,}\\
  0\;\; \text{or} \;\; 1  & \text{otherwise.}
\end{cases}
\]
By \eqref{eq:biquad}, we obtain the values of $H_F(4)$ and $H_F(3)$.
The case of $H_F(2)$ when $D=8$ is similar, and we obtain $H_F(2)=1/4$, which yields the term $I(8)$.
In the case $D=5$, by \cite{lmfdb:4.0.125.1}, 
we obtain $H_F\left(\frac{5\pm \sqrt{5}}{2}\right)=\frac{1}{5}$, which yields the term $I(5)$.
\end{proof}
\begin{remark}
A dimension formula for the space of Hilbert cusp forms over a real quadratic field, such as Corollary \ref{cor:main1}, is known to be computable by combining the trace formulas for dimensions \cites{Shi63,Ish74} with the information on torsion elements of $\SL_2(\fo_F)$ given in \cite{Pre68}.
For example, see \cite{Ish88} for the formula and numerical tables for the dimension $\dim S_{(2,2)}(\SL_2(\fo_F))$.
\end{remark}

\begin{corollary}\label{cor:main2}
Among the fundamental discriminants $D$ satisfying the above conditions, we have
$\dim S_{(2,2)}(\SL_2(\fo_F))=0$ only for $D=5$, $8$, $13$, $17$.
In these cases, an analogue of the class number relation for Hurwitz class numbers holds as follows.
\begin{equation}\label{eq:relation}
 \sum_{t\in \fo_F} H_F(4n-t^2) =      2 \sum_{\fa\mid(n)}N_{F/\Q}(\fa)  
\end{equation}
\end{corollary}
\begin{remark}
Of course, this class number relation is analogous to the class number relation for $H(n)$ obtained from the case $k=2$ of \eqref{eq:3}. 
In \cite{JRW06}, the class number relation was used to verify numerical values of class numbers obtained under the assumption of the GRH. 
This class number relation may also be applicable to the verification of class number computations, although we do not pursue such an application in this paper and leave this for future work.
\end{remark}
\begin{proof}
    For $D\le 1000$ satisfying the above conditions, it can be verified from the numerical tables in \cite{Ish88} that the only values of $D$ for which $\dim S_{(2,2)}(\SL_2(\fo_F))=0$ are $5$, $8$, $13$, and $17$.
    By Corollary \ref{cor:main1} we have $\dim S_{(2,2)}(\SL_2(\fo_F))\ge \frac{1}{48}B_{2,\chi_D}-1$, and using the Dirichlet $L$-function $L(s,\chi_D)$ we obtain
    \[
    \frac{1}{48}B_{2,\chi_D}=-\frac{1}{24}L(-1,\chi_D)=\frac{D^{\frac32}}{96\pi^2}L(2,\chi_D)\ge \frac{D^{\frac32}}{96\pi^2}\frac{\zeta(4)}{\zeta(2)}=\frac{D^{\frac32}}{1440}.
    \]
    Hence, for $D> 1000$ we obtain $\dim S_{(2,2)}(\SL_2(\fo_F))>0$.
\end{proof}

\begin{corollary}\label{cor:main3}
We have that
    \[
    \sum_{\substack{n\mod U_F, \\ N_{F/\Q}(n)<X}}\sum_{t\in\fo_F}H_F(4n-t^2) \sim L(1,\chi_D) \, \zeta_F(2)\,  X^2 ,\qquad X\to\infty. 
    \]
    Here, $n$ runs over all totally positive elements of $\fo_F$ with $N_{F/\Q}(n)<X$, taken modulo $U_F$.
\end{corollary}
\begin{proof}
    Let $d_F\coloneqq\dim S_{(2,2)}(\SL_2(\fo_F))$, and let $\{ f_j \}_{1\le j\le d_F}$ be a basis of $S_{(2,2)}(\SL_2(\fo_F))$ consisting of Hecke eigenforms. Then, by Theorem \ref{thm:ES} for $\kappa_1=\kappa_2=2$, we have
    \[
    \frac{1}{2}\sum_{n\mod U_F} \frac{\sum_{t\in\fo_F}H_F(4n-t^2) }{N_{F/\Q}(n)^{s-\frac{1}{2}}}=\sum_{j=1}^{d_F} L(f_j,s)+\zeta_F(s-\tfrac{1}{2})\, \zeta_F(s-\tfrac{3}{2}).
    \]
    Here, $n$ runs over all totally positive elements of $\fo_F$ modulo $U_F$, and $L(f_j,s)$ denotes the Hecke $L$-function associated with $f_j$ (cf. \cite{Gar90}*{\S1.9}).
    Since $L(f_j,s)$ is entire on $\C$, the assertion follows from the Tauberian theorem.
\end{proof}
\begin{remark}
Note that there exists a constant $c>0$ such that
\begin{equation*}\label{eq:asympbypv}
\sum_{\substack{n\mod U_F,\\ N_{F/\Q}(n)<X}}H_F(n) \sim c\, X^{\frac{3}{2}}, \qquad X\to\infty.    
\end{equation*}
This asymptotic formula can be proved by using the explicit formulas for the prehomogeneous zeta functions of the space of binary quadratic forms (cf. \cite{Dat93}*{\S5} and \cite{HW18}*{Ch.4}) together with a Tauberian theorem.
For the fact that the zeta function is given by the Dirichlet series of $H_F(n)$, see \cite{KTW22}*{\S5}.
\end{remark}
\begin{corollary}\label{cor:main4}
    For any arbitrarily small $\epsilon>0$, 
    \[
    \sum_{t\in\fo_F} H_F(4n-t^2) \ll_\epsilon N_{F/\Q}(n)^{1+\epsilon}.    
    \]
\end{corollary}
\begin{proof}
    The eigenvalues of $\bT_n$ admit the trivial upper bound $\sum_{\fa\mid (n)}N_{F/\Q}(\fa)$. 
In other words, by Theorem \ref{thm:ES} with $\kappa_1=\kappa_2=2$, it follows that
\[
\sum_{t\in\fo_F}H_F(4n-t^2) \le 2(\dim S_{(2,2)}(\SL_2(\fo_F))+1) \sum_{\fa\mid (n)}N_{F/\Q}(\fa).
\]
This completes the proof.
\end{proof}

\subsection{The distribution of generalized Hurwitz class numbers \texorpdfstring{$H_F(n)$}{HF(n)}}
In Corollary \ref{cor:main4}, we established an upper bound for $\sum_{t\in\fo_F} H_F(4n-t^2)$. In this paper, we further derive an explicit trace formula for the resolvent of Hecke operators from Theorem \ref{thm:ES}, and thereby prove an optimal estimate and an equidistribution theorem for $\sum_{t\in\fo_F} H_F(4p^\nu-t^2)$, where $p$ is a totally positive prime element in $\fo_F$.
For a totally positive prime element $p\in\fo_F$ and $X\in\C$ , we now consider the following operator on $S_{\underline{\kappa}}(\SL_2(\fo_F))$:
\begin{align*}
    (\mathbb{T}_p'-(X+X^{-1})\id)^{-1}
\end{align*} 
where we set
\begin{equation}\label{eq:bTn'bTn}
\mathbb{T}_n'\coloneqq N_{F/\Q}(n)^{-\frac{1}{2}}\mathbb{T}_n    
\end{equation}
for $n\in\fo_F$, $n\gg 0$. 
This makes sense if $|X|$ is sufficiently small because $\dim S_{\underline{\kappa}}(\SL_2(\fo_F))<\infty$. (In fact, by a result on the generalized Ramanujan conjecture for Hilbert modular forms, it can be defined for $|X|<1$, which will be mentioned later in Theorem \ref{thm:GRC}.)
\begin{theorem}\label{RTF}
Let $p$ be a totally positive prime element in $\fo_F$. For any $X\in\C$ such that $|X|$ is sufficiently small, we have
\begin{align*}
    \tr\, (\mathbb{T}_p'-(X+X^{-1}))^{-1}=&
    -\frac{1}{48}B_{2,\chi_D}(\kappa_1-1)(\kappa_2-1)\frac{X}{1-q^{-1}X^2}\\
    &-\frac{X}{2}\sum_{n=0}^{\infty}A_{\underline{\kappa}}(p^{n})(q^{-\frac{1}{2}}X)^n\\
    &+\begin{cases}
    \frac{-X}{(1-q^{\frac{1}{2}}X)(1-q^{-\frac{1}{2}}X)}\quad &if\ \kappa_1=\kappa_2=2,\\
    0 & otherwise,
    \end{cases}
\end{align*}
where $q:=N_{F/\Q}(p)$ and
$
A_{\underline{\kappa}}(p^n):=\sum_{4p^n-t^2\gg0} H_F(4p^n-t^2)\, U_{\kappa_1-2}\left( \frac{t}{2p^{\frac{n}{2}}} \right)\, U_{\kappa_2-2}\left( \frac{\sigma(t)}{2\sigma(p)^{\frac{n}{2}}} \right)
$.
\end{theorem}
As an application of the resolvent trace formula, we establish an optimal estimate for a weighted average of the generalized Hurwitz class numbers.

Let us define a measure $\mu$ on the region $\cD:=[-1,1]\times[-1,1]$ by
$$
\langle\mu,f\rangle:=\frac{4}{\pi^2}\int_{-1}^{1}\int_{-1}^{1}f(x,y)\sqrt{1-x^2}\sqrt{1-y^2}\,dx\,dy,\quad f\in C^{0}(\cD),
$$
and set two subspaces of polynomial rings as
$$
V=\{f\in\C[x,y]\mid f(x,y)=f(x,-y)=f(-x,y)\},\quad V_0:=\{f\in V\mid \langle\mu,f\rangle=0\}.
$$
By the orthogonality of the Chebyshev polynomials, the subspace $V_0$ can also be identified with the subspace ${\rm Span}\{U_{l_1}(x)U_{l_2}(y)\mid (l_1,l_2)\in(\Z_{\ge0})^2-\{(0,0)\}\}$. 
For $f\in V$ and $\alpha>0$, let us consider the following statement:
$$
(E_{f,\alpha}):\forall\epsilon>0,\quad 
\sup_{\nu\in\N}\left|
q^{-\nu(\alpha+\epsilon)}\sum_{\substack{t\in\fo_F\\4p^{\nu}-t^2\gg0}}
H_F(4p^{\nu}-t^2)f\left(\tfrac{t}{2\sqrt{4p^\nu}},\tfrac{\sigma(t)}{2\sqrt{\sigma(4p^\nu)}}\right)
\right|<+\infty.
$$

A natural question to ask here is: what is the optimal choice of $\alpha$?
The following statement answers this question and may be viewed as a generalization of the result for classical Hurwitz class numbers \cite{ST 20}*{Theorem 3}.

\begin{corollary}\label{Optimal}
Let $p\in\fo_F$ be a totally positive prime element and $\alpha>0$.
\begin{itemize}
    \item[(i)] If $\alpha\ge\frac{1}{2}$, the statement $(E_{f,\alpha})$ holds for any $f\in V_0$. If $0<\alpha<\frac{1}{2}$, there exists $f\in V_0$ such that the statement $(E_{f,\alpha})$ does not hold.
    \item[(ii)] If $\alpha\ge1$, the statement $(E_{f,\alpha})$ holds for any $f\in V$. If $0<\alpha<1$, there exists $f\in V-V_0$ such that the statement $(E_{f,\alpha})$ does not hold. 
\end{itemize} 
\end{corollary}
We conclude this section by stating the equidistribution theorem for the generalized Hurwitz class numbers. Let $\cD\coloneqq [-1,1]\times[-1,1]$ and $C_{\rm e}^{0}(\cD)$ denote the set of all continuous function on $\cD$ satisfying $f(x,y)=f(-x,y)=f(x,-y)$.
For a totally positive prime element $p\in\fo_F$ and $\nu\in\N$, let us consider the discrete measure $\mu_{p,\nu}$ on $\cD$ given by
\begin{align*}
\langle\mu_{p,\nu},f\rangle:=\frac{1-q^{-1}}{2q^{\nu}}\sum_{\substack{t\in\fo_F\\ 4p^{\nu}-t^2\gg 0}}H_F(4p^{\nu}-t^2)f\left(\tfrac{t}{2\sqrt{4p^\nu}},\tfrac{\sigma(t)}{2\sqrt{\sigma(4p^\nu)}}\right),\quad f\in C_{\rm e}^{0}[\cD].
\end{align*}
\begin{theorem}\label{Equid}
The measure $\mu_{p,\nu}$ converges $*$-weakly to the measure $\mu$ as $\nu\to\infty$. i.e.,
$$
\lim_{\nu\to\infty}\langle\mu_{p,\nu},f\rangle=\langle\mu,f\rangle
$$
holds for any $f\in C_{\rm e}^{0}(D)$.
\end{theorem}

\section{Notation}
First, we fix the notation related to totally real number fields as follows.
\begin{itemize}
    \item Let $F$ be a totally real number field. 
    \item The degree of $F$ over $\Q$ is $m$, and $\iota_1$, $\iota_2,\dots,\iota_m$ are the $m$ embeddings of $F$ into $\R$.  
    \item We suppose that $m\ge 2$ throughout this paper. 
    \item Write $\fo_F$ for the ring of integers of $F$.
    \item Let $\A$ denote the adele ring of $F$, and $\A_f$ the finite adele ring of $F$.
    \item For each place $v$ of $F$, let $F_v$ denote the completion of $F$ at $v$.
    \item For each finite place $v$ of $F$, let $\fo_{F_v}$ denote the ring of integers of $F_v$. 
    \item Set $F_\inf\coloneqq\prod_{v\mid\inf}F_v$.
    \item Via the diagonal embedding, we regard $F$ as a subring of $\A$. In particular, via the embedding $a \mapsto (\iota_1(a),\dots,\iota_m(a))$, the field $F$ becomes a subring of $\R^m$.
    \item For $a \in F$, we define the norm as follows. 
    \[
    N_{F/\Q}(a)\coloneqq\prod_{j=1}^m\iota_j(a) . 
    \]
    \item Let $a\in F\setminus \{0\}$. The notation $a \gg 0$ means that $a$ is totally positive, that is, $\iota_1(a), \dots, \iota_m(a)$ are all positive real numbers.
\end{itemize}
Next, let us introduce the notation related to Hilbert modular forms.
\begin{itemize}
    \item $\fh\coloneqq\{z\in\C\mid \mathrm{Im}(z)>0\}$.
    \item $\GL_2(\R)^+\coloneqq\{ g\in\GL_2(\R)\mid \det(g)>0\}$. 
    \item $\GL_2(\R)^+$ acts on $\fh$ as $g\cdot z=(az+b)(cz+d)^{-1}$, $g=\begin{pmatrix}
        a&b \\ c&d
    \end{pmatrix}\in\GL_2(\R)^+$, $z\in\fh$. 
    \item $\prod_{v\mid \inf}\GL_2(F_v)^+=(\GL_2(\R)^+)^m$ acts on $\fh^m$ as $(g_j)_{1\le j\le m}\cdot (z_j)_{1\le j\le m}= (g_j\cdots z_j)_{1\le j\le m}$. 
    \item Via the diagonal embedding, $\GL_2(F)$ becomes a discrete subgroup of $\GL_2(\A)$. In particular, under the embedding $\gamma \longmapsto (\iota_1(\gamma), \dots, \iota_m(\gamma))$, 
    the group $\GL_2(F)$ is realized as a subgroup of $\GL_2(\R)^m$.
    \item For $g=\begin{pmatrix}
        a&b\\c&d
    \end{pmatrix}\in\GL_2(\R)^+$ and $z\in\fh$, we set $J(g,z)\coloneqq \det(g)^{-1/2}(cz+d)$. 
    \item $\underline{\kappa}=(\kappa_1,\dots,\kappa_m)\in (\Z_{\ge 2})^m$. 
    \item For $g=(g_1,\dots,g_m)\in(\GL_2(\R)^+)^m$ and $z=(z_1,\dots,z_m)\in\fh^m$, we set 
    \[
    J_{\underline{\kappa}} (g,z)\coloneqq \prod_{j=1}^m J(g_j,z_j)^{\kappa_j}.
    \]
    \item $\Gamma\coloneqq \GL_2(\fo_F)\cap (\GL_2(\R)^+)^m$ is regarded as a discrete subgroup of $(\GL_2(\R)^+)^m$ via $\gamma\mapsto (\iota_1(\gamma),\dots,\iota_m(\gamma))$. 
\end{itemize}
Let $\tilde{\Gamma}$ be a subgroup of $\Gamma$ of finite index.
Let $S_{\underline{\kappa}}(\tilde{\Gamma})$ denote the space of functions $\phi$ on $\fh^m$, which satisfy the following conditions:
\begin{itemize}
    \item $\phi$ is holomorphic on $\fh^m$. 
    \item For any $\gamma\in\tilde{\Gamma}$ and $z\in\fh^m$, we have
    \[
    \phi(\gamma\cdot z)=J_{\underline{\kappa}}(\gamma,z)\, \phi(z).
    \]
    \item The function $\left| \prod_{j=1}^m\mathrm{Im}(z_j)^{\frac{\kappa_j}{2}}\times \phi(z) \right|$ is bounded on $\fh^m$.
\end{itemize}
The space $S_{\underline{\kappa}}(\tilde{\Gamma})$ is the space of so-called holomorphic Hilbert cusp forms.

We further fix the following notation for algebraic groups.
\begin{itemize}
  \item Let $\rG=\GL_2$ denote the general linear group of degree $2$ over $F$.
  \item Let $Z$ denote the center of $\rG$, that is, $Z\simeq \mathbb{G}$. 
  \item Set $\overline{\rG}\coloneqq\rG/Z$, that is, $\overline{\rG}$ is the projective linear group of degree $2$ over $F$.
  \item For $x\in\rG(\A)$, we denote by $\overline{x}$ the projection of $x$ to $\overline{\rG}(\A)$. 
  \item For any subset $H$ of $\rG(\A)$, we also denote by $\overline{H}$ the image of $H$ via the projection. 
  \item For each finite place $v<\infty$, set $K_v\coloneqq\GL_2(\fo_{F_v})$, and define
  $K_f \coloneqq \prod_{v<\infty} K_v$.
  \item Let $C_c^\infty(K_f\bs G(\A_f)/K_f)$ denote the space of locally constant,
  compactly supported functions on $G(\A_f)$ that are bi-$K_f$-invariant.
  \item We choose the Haar measure $dg$ on $G(\A_f)$ so that $\int_{K_f} dg = 1$. 
\end{itemize}

\section{Conditions and consequences of known results}

\subsection{Conditions}

In what follows, we consider the following two conditions, depending on the context. 
\begin{itemize}
    \item[(I)] The (wide) class number of $F$ is $1$.
    \item[(II)] The degree $m$ of $F$ over $\Q$ is $2$. The fundamental unit $\varepsilon$ of $\fo_F^\times$ satisfies $N_{F/\Q}(\varepsilon)=-1$. 
\end{itemize}
Conditions (I) and (II) are satisfied if and only if $F$ is a real quadratic field with narrow class number $1$.

Under the assumption of Condition (I), we have
\begin{equation}\label{eq:I1}
\A^\times =F^\times \, F_\inf^\times U_f, \quad \text{where $U_f\coloneqq \prod_{v<\inf}\fo_{F_v}^\times$}.    
\end{equation}
In addition, by \eqref{eq:I1} together with the strong approximation theorem for $\SL_2$, we have
\begin{equation}\label{eq:I2}
\rG(\A)=\rG(F)\rG(F_\inf)K_f. 
\end{equation}
By \eqref{eq:I1} and \eqref{eq:I2} we obtain a diffeomorphism
\begin{equation}\label{eq:diff1}
\overline{\rG}(F)\bs \overline{\rG}(\A) /\overline{K_f} \simeq  \rG(F)\bs \rG(\A) /Z_\inf K_f \simeq \rG(\fo_F)\bs \rG(\R)^2 /Z_\inf,     
\end{equation}
where $Z_\inf\coloneqq Z(F_\inf)$.

In addition to Condition (I), we also assume Condition (II), which implies
\[
\varepsilon I_2,\quad (-\varepsilon)I_2, \quad 
\delta_\varepsilon, \quad \delta_{-\varepsilon}
\in \GL_2(\fo_F), \quad \iota_1(\varepsilon)\iota_2(\varepsilon)<0,
\] 
where we set
\[
\delta_a\coloneqq\begin{pmatrix}
    1&0 \\ 0&a
\end{pmatrix}\in \rG(F), \quad a\in F^\times.
\]
\begin{lemma} Recall $\Gamma\coloneqq \rG(\fo_F)\cap (\GL_2(\R)^+)^2$ and set $\Gamma_1\coloneqq\SL_2(\fo_F)$. 
It is obvious that $S_{\underline{\kappa}}(\Gamma)$ is a subspace of $S_{\underline{\kappa}}(\Gamma_1)$. Furthermore, we have:
    \begin{itemize}\label{lem:vanishing}
        \item[(1)] If $\kappa_1+\kappa_2$ is odd, then $S_{\underline{\kappa}}(\Gamma_1)=0$. 
        \item[(2)] If $\kappa_1$ and $\kappa_2$ are odd, then $S_{\underline{\kappa}}(\Gamma)=0$. 
        \item[(3)] If $\kappa_1$ and $\kappa_2$ are even, then $S_{\underline{\kappa}}(\Gamma)=S_{\underline{\kappa}}(\Gamma_1)$. 
    \end{itemize}
\end{lemma}
    \begin{proof}
        This follows from the definition and from the actions of $\delta_\varepsilon$ and $\varepsilon I_2$.
    \end{proof}

Using $\delta_{\pm\varepsilon}$ and $\det\colon G\to \mathbb{G}_m$, we obtain
\begin{equation}\label{eq:diff2}
\rG(F)\backslash \rG(\A) /Z_\inf K_f\simeq \SL_2(F)\bs \SL_2(\A)/ Z_{1,\infty} K_{1,f} \simeq  \overline{\Gamma_1} \backslash \PSL_2(\R)^2,  
\end{equation}
where $Z_{1,\infty}\coloneqq \{(\pm I_2,\pm I_2)\in Z(F_\infty)\}$,  $K_{1,f}\coloneqq\prod_{v<\infty}\SL_2(\fo_{F_v})$, $\overline{\Gamma_1}\coloneqq \Gamma_1/\{\pm I_2\}$, and $\PSL_2(\R)\coloneqq\SL_2(\R)/\{\pm I_2\}$. 
This diffeomorphism implies that the interpretation in terms of $\rG(\A)$ can be
applied precisely to holomorphic Hilbert cusp forms. 
A holomorphic Hilbert cusp form $\phi\in S_{\underline{\kappa}}(\Gamma)$ is, by \eqref{eq:diff2}, defined as a function $F_\phi$ on $\rG(F)\bs \rG(\A)/Z_\infty K_f$ as follows:
\begin{equation}\label{eq:lift}
F_\phi(z\gamma g k)\coloneqq   J_{\underline{\kappa}}( g,\underline{i})^{-1} \, \phi(g\cdot \underline{i}), \quad  z\in Z_\infty,\;\; \gamma\in \rG(F), \;\; g\in (\GL_2(\R)^+)^2, \;\; k\in K_f,    
\end{equation}
where $\underline{i}=(i,i)$ $(i=\sqrt{-1})$. 
Note that $\delta_\varepsilon\cdot \underline{i}\notin\fh^2$ and we suppose $\phi(z)=0$ if $z\notin\fh^2$. 
We identify $\phi$ with $F_\phi$, and then $S_{\underline{\kappa}}(\tilde{\Gamma})$ can be regarded as a sub-Hilbert space of $L^2(\rG(F)\bs \rG(\A) /Z_\inf K_f)$.
If $\kappa_1$ and $\kappa_2$ are odd, then any non-zero element of $S_{\underline{\kappa}}(\Gamma_1)$ can not be regarded as a function in $L^2(\rG(F)\bs \rG(\A) /Z_\inf K_f)$, since $J_{\underline{\kappa}}((\iota_1(\varepsilon)I_2,\iota_2(\varepsilon)I_2)g,z)=-J_{\underline{\kappa}}(g,z)$. 

\subsection{Adelic Hecke operator}\label{sec:Hecke operators}

We review Hecke operators.
For details, we refer to \cite{Gar90}.

Assume Condition {\rm (I)}, that is, every ideal of $\fo_F$ is principal.
Take an element $n\in \fo_F$. 
We define the set $\fT_n$ by 
\[
\fT_n \coloneqq \{ x \in \M_2(\fo_F)\cap \rG(F) \mid \det(x)\in \fo_F^\times\, n \}  .
\]
By the relation
\begin{equation}\label{eq:hecke1}
\fT_n = \rG(F_\inf) \cap (\rG(F) \, T_{n,\A_f}),    
\end{equation}
we define the compact subset $T_{n,\A_f}$ of $G(\A_f)$. 
Let $h_n$ denote the characteristic function of $T_{n,\A_f}$ on $\rG(\A_f)$. 
Note that $h_n$ is bi-$K_f$-invariant, that is, $h_n \in C_c^\inf(K_f\bs \rG(\A_f)/K_f)$. 
For the adelic interpretation of Hecke operators, we introduce 
the following integral operator $R_\rG(h_n)$ on
$L^2(\rG(F)\bs \rG(\A) /Z_\inf K_f)$:
\[
(R_\rG(h_n)F)(x)\coloneqq\int_{\rG(\A_f)} h_n(g) \, F(xg) \, dg
,\quad F\in L^2(\rG(F)\bs \rG(\A) /Z_\inf K_f). 
\]

\subsection{Hecke eigenvalue and Ramanujan conjecture}\label{sec:discrete Hecke}

We assume Conditions (I) and (II). 
Under the assumptions, for any principal ideal $(n)$, we may suppose $n\gg 0$ without loss of generality.
Set
\[
\mathbf{T}_n \coloneqq \{ x \in \M_2(\fo_F) \mid \det(x)= n \} \quad (\subset(\GL_2(\R)^+)^2).
\]
The set $\mathbf{T}_n$ defines a Hecke operator $\bT_n$ on
$S_{\underline{\kappa}}(\Gamma)$ by
\begin{equation}\label{eq:bTn}
(\bT_n \phi)(z)
\coloneqq
\sum_{\gamma \in \Gamma \backslash \mathbf{T}_n}
J_{\underline{\kappa}}(\gamma,z)^{-1} \phi(\gamma \cdot z),
\qquad
\phi \in S_{\underline{\kappa}}(\Gamma).    
\end{equation}
\begin{lemma}\label{lem:hecke1}
In this setting, the following holds as an adelic interpretation of Hecke operators:
\[
R_\rG(h_n)F_\phi = F_{\bT_n\phi}.
\]    
\end{lemma}
\begin{proof}
Write $\mathrm{pr}_f$ (resp. $\mathrm{pr}_\inf$) the projection from $\rG(\A)$ to $\rG(\A_f)$ (resp. $\rG(F_\inf)$). 
It follows from \eqref{eq:I2} that there are elements $\beta_1,\dots,\beta_w\in \rG(F)\cap \M_2(\fo_F)$ such that
\[
T_{n,\A_f}=\bigsqcup_{t=1}^w  \mathrm{pr}_f(\beta_t^{-1})\, K_f \quad \text{and} \quad \det(\beta_t)\in (n). 
\]
By Condition (II), we may suppose $\beta_1,\dots,\beta_w\in \rG(F)\cap (\GL_2(\R)^+)^2$ without loss of generality. 
On the other hand, by \eqref{eq:hecke1},
\[
\fT_n = \bigsqcup_{t=1}^w \rG(\fo_F) \, \mathrm{pr}_\inf(\beta_t), 
\]
which implies $\Gamma\bs\mathbf{T}_n = \bigsqcup_{t=1}^w\Gamma\beta_t$. 
Thus, for $x\in (\GL_2(\R)^+)^2$,
\begin{multline*}
(R_\rG(h_n)F_\phi)(x)=\sum_{t=1}^w F_\phi(x\, \mathrm{pr}_f(\beta_t^{-1}))=\sum_{t=1}^w F_\phi(\mathrm{pr}_\inf(\beta_t)\, x)=\sum_{t=1}^w J_{\underline{\kappa}}(\beta_t x,\underline{i})^{-1}\phi(\beta_t x\cdot \underline{i}) \\
= J_{\underline{\kappa}}(x,\underline{i})^{-1} \sum_{t=1}^w J_{\underline{\kappa}}(\beta_t ,x\cdot\underline{i})^{-1}\phi(\beta_t\cdot (x\cdot\underline{i}))=J_{\underline{\kappa}}(x,\underline{i})^{-1} (\bT_n \phi)(x\cdot\underline{i})=F_{\bT_n \phi}(x).
\end{multline*}
Hence, this completes the proof. 
\end{proof}

It is known that the Hecke operators $\bT_n$ on $S_{\underline{\kappa}}(\Gamma)$ are commutative and simultaneously diagonalizable for all prime ideals $(n)$. Their eigenfunctions are called Hecke eigenforms, and $S_{\underline{\kappa}}(\Gamma)$ admits a basis consisting of Hecke eigenforms.
\begin{theorem}[The generalized Ramanujan conjecture]\label{thm:GRC}
Let $\phi$ be a Hecke eigenform in $S_{\underline{\kappa}}(\Gamma)$, and denote by $\lambda_\phi(n)$ the eigenvalue of $\phi$ with respect to $\bT_n'= N_{F/\Q}(n)^{-1/2}\bT_n$ for each ideal $(n)$. 
Then, we have $|\lambda_\phi(p)|\le 2$ for each prime ideal $(p)$. 
\end{theorem}
\begin{proof}
    Since $S_{\underline{\kappa}}(\Gamma)=0$ when $\kappa_1$ or $\kappa_2$ is odd (see Lemma \ref{lem:vanishing}), the assertion of the theorem is established in \cite{Bla06}*{Theorem 1}.
    Moreover, in \cite{Bla06}*{Theorem 1}, the assertion is formulated in terms of automorphic representations, but by Lemma \ref{lem:hecke1} the generalized Ramanujan conjecture for the automorphic representation generated by $\phi$ can be reformulated as the corresponding statement on $\lambda_\phi(p)$.
\end{proof}

\subsection{Orbital integral}

Suppose that Condition (I) holds. 
Take a regular semisimple element $\overline{\gamma}\in  \overline{\rG}(F)$. 
We write $\overline{\rG}_{\overline{\gamma}}$ for the centralizer of $\overline{\gamma}$ in $\overline{\rG}$, that is,
\[
\overline{\rG}_{\overline{\gamma}}\coloneqq\{ g\in \overline{\rG} \mid g{\overline{\gamma}}={\overline{\gamma}} g  \}.
\]
For $h\in C_c^\inf(\rG(\A_f))$, we set
\[
\overline{h}(g)\coloneqq \int_{Z(\A_f)} h(zg)\, \rd z \in C_c^\inf(\overline{\rG}(\A_f)).
\]
Choose a Haar measure $\rd g_{\overline{\gamma}}$ on $\overline{\rG}_{\overline{\gamma}}(\A_f)$ and set
\[
J({\overline{\gamma}},\overline{h})\coloneqq \int_{\overline{\rG}_{\overline{\gamma}}(\A_f)\bs \overline{\rG}(\A_f)} \overline{h}(g^{-1}{\overline{\gamma}} g) \frac{\rd g}{\rd g_{\overline{\gamma}}},\quad h\in C_c^\inf(\rG(\A_f)),
\]
where $\frac{\rd g}{\rd g_{\overline{\gamma}}}$ is the quotient measure on $\overline{\rG}_{\overline{\gamma}}(\A_f)\bs \overline{\rG}(\A_f)$.

Suppose 
that ${\overline{\gamma}}\in \overline{G}(F)$ is regular and $\R$-elliptic. 
Since $\overline{\rG}_{\overline{\gamma}}(F_\inf)$ is compact, $\overline{\rG}_{\overline{\gamma}}(F)$ is a discrete subgroup of $\overline{\rG}_{\overline{\gamma}}(\A_f)$. 
Hence, there exists a fundamental domain $\mathcal{F}$ in $\overline{\rG}_{\overline{\gamma}}(\A_f)$ such that
\[
\int_{\overline{\rG}_{\overline{\gamma}}(F)\bs \overline{\rG}_{\overline{\gamma}}(\A)} \rd g_{\overline{\gamma}} = \int_{\overline{\rG}_{\overline{\gamma}}(F_\inf)} \rd g_{\overline{\gamma},\inf} \int_{\mathcal{F}} \rd g_{\overline{\gamma},f}, \quad  g_{\overline{\gamma}}=(g_{\overline{\gamma},\inf},g_{\overline{\gamma},f})\in \overline{\rG}_{\overline{\gamma}}(\A),
\]
cf. \cite{Bou04}*{Ch. 7} (see also \cite{KL06}*{\S26.1}). 
We normalize $\int_{ \overline{\rG}_{\overline{\gamma}}(F_\inf)} \rd g_{\overline{\gamma},\inf}=1$. 
Then,
\begin{align}\label{eq:JPhi}
& \vol(\overline{\rG}_{\overline{\gamma}}(F)\bs \overline{\rG}_{\overline{\gamma}}(\A))\, J({\overline{\gamma}},\overline{h}) =\int_{\overline{\rG}_{\overline{\gamma}}(F)\bs \overline{\rG}_{\overline{\gamma}}(\A)}\rd g_{\overline{\gamma}}\, \int_{\overline{\rG}_{\overline{\gamma}}(\A)\bs \overline{\rG}(\A_f)} \overline{h}(g_f^{-1}{\overline{\gamma}} g_f)\,  \frac{\rd g_f}{\rd (g_{\overline{\gamma}})_f} \\
& = \int_{ \overline{\rG}_{\overline{\gamma}}(F_\inf)} \rd g_{\overline{\gamma},\inf}\int_{\overline{\rG}_{\overline{\gamma}}(F)\bs \overline{\rG}(\A_f)} \overline{h}(g_f^{-1}{\overline{\gamma}} g_f)\,  \rd g_f = \int_{\overline{\rG}_{\overline{\gamma}}(F)\bs \overline{\rG}(\A_f)} \overline{h}(g_f^{-1}{\overline{\gamma}} g_f)\, \rd g_f= \overline{\Phi}({\overline{\gamma}},\overline{h}) \nonumber 
\end{align}
where

\begin{align}
\overline{\Phi}({\overline{\gamma}},\overline{h})\coloneqq \int_{\overline{\rG}_{\overline{\gamma}}(F)\bs \overline{\rG}(\A_f)}\overline{h}(g^{-1}{\overline{\gamma}} g) \, \rd g.\label{OrbitalInt}
\end{align}

\subsection{Traces of Hecke operators}

Assume that Conditions (I) and (II) hold. 
If $\kappa_1$ and $\kappa_2$ are even, then we can get a formula of $\tr\, \mathbb{T}_n|_{S_{\underline{\kappa}}(\Gamma)}$ applying \cite{Art89}*{Theorem 6.1} to $\overline{\rG}$. 
We then extend the formula to arbitrary $\kappa_1$ and $\kappa_2$ by supplementing the cases in which it is identically zero.

We recall discrete series representations and their characters. 
Let $D_\kappa$ denote the discrete series of $\overline{\rG}(\R)$ whose Harish-Chandra parameter is $\kappa-1$, where $\kappa\in2\Z_{\ge 1}$. 
Consider the compact torus $T$ defined as
\[
T\coloneqq \left\{ \overline{\gamma_\theta} \; \middle| \; \theta\in\R \right\},\quad \gamma_\theta\coloneqq\begin{pmatrix}
    \cos(\theta) & -\sin(\theta) \\
    \sin(\theta) & \cos(\theta)
\end{pmatrix}.
\]
If $\sin(\theta)\neq 0$, then the character value of $D_\kappa$ for $\overline{\gamma_\theta}$ agrees with
\[
U_{\kappa-2}(\cos(\theta))=\frac{\sin((\kappa-1)\theta)}{\sin(\theta)}
\]
where $U_l$ is the $l$-th Chebyshev polynomial of the second kind. 

Write $m_\pi$ for the multiplicity of the discrete spectrum of the $L^2$-space of $\overline{\rG}(F)\bs \overline{\rG}(\A)/\overline{K_f}$. 
For $h\in C_c^\inf(K_f\bs \rG(\A_f)/K_f)$ and $\underline{\kappa}=(\kappa_1,\kappa_2)\in (2\Z_{\ge 1})^2$, we set
\[
\cL(\underline{\kappa},h)\coloneqq \sum_{\pi} m_\pi \, \tr(\pi_f(\overline{h})),
\]
where $\pi=\pi_\inf\otimes\pi_f$ runs over automorphic representations of $\overline{\rG}(\A)$ whose archimedean component $\pi_\inf$ is isomorphic to $D_{\kappa_1}\otimes D_{\kappa_2}$. 
We set
\[
\mathscr{R}(\underline{\kappa},h)\coloneqq \begin{cases}
    \int_{\overline{\rG}(\A_f)} \overline{h}(g) \, \rd g & \text{if $\kappa_1=\kappa_2=2$,} \\
    0 & \text{otherwise.}
\end{cases}
\]
\begin{lemma}\label{lem:31}
If $\kappa_1$ and $\kappa_2$ are even, for any $h\in C_c^\inf(K_f\bs \rG(\A_f)/K_f)$, we obtain
\begin{align*}
\cL(\underline{\kappa},h) + \mathscr{R}(\underline{\kappa},h)  =&\frac{1}{2} \zeta_F(-1)\, (\kappa_1-1)(\kappa_2-1) \, \overline{h}(1)  \\
& + \sum_{\{{\overline{\gamma}}\}_{\overline{\rG}(F)}}\vol(\overline{\rG}_{\overline{\gamma}}(F)\bs \overline{\rG}_{\overline{\gamma}}(\A))\,  J({\overline{\gamma}},\overline{h})\, U_{\kappa_1-2}( \cos(t_1))\, U_{\kappa_2-2}(\cos(t_2))  ,  \nonumber
\end{align*}
where $\{{\overline{\gamma}}\}_{\overline{\rG}(F)}$ runs over all regular $\R$-elliptic $\overline{\rG}(F)$-conjugacy classes and $t_j$ $(j=1,2)$ is defined so that $\iota_j({\overline{\gamma}})$ is $\overline{\rG}(\R)$-conjugate to $\overline{\gamma_{t_j}}$. 
\end{lemma}
\begin{proof}
    See \cite{Art89}*{Theorem 6.1}.
    The factor $\zeta_F(-1)/2$ arises from the volume $\vol(\overline{G}(F)\bs \overline{G}(\A))$.
    For details, see \cite{Shi63}*{(53) on p.~70}.

    $\mathscr{R}(\underline{\kappa},h)$ is the trace of the action of $h$ for the trivial representation. The reason why such a term appears is that the pseudo-coefficient of $D_2$ takes a nonzero value on the trivial representation, cf. \cite{CD90}. 
    By using \eqref{eq:diff2} and the strong approximation theorem \cite{PR94}*{Theorem 7.12} for $\SL_2$, one can prove that an additional contribution to the spectral side, other than $\cL(\underline{\kappa},h)$, appears only in the case $\kappa_1=\kappa_2=2$.
\end{proof}

\begin{lemma}\label{lem:32} For any $n\in\fo_F$, $n\gg 0$, we have $\cL(\underline{\kappa},h_n)=\tr \bT_n|_{S_{\underline{\kappa}}(\Gamma)}$. 
\end{lemma}
\begin{proof}
This follows from \eqref{eq:diff1}, \eqref{eq:diff2}, Lemmas \ref{lem:vanishing} and \ref{lem:hecke1}, since $h_n$ is $K_f$-bi-invariant. 
\end{proof}

\begin{theorem}\label{thm:traceformula}
Suppose that {\rm Conditions (I) and (II)} hold. 
Let $\zeta_F(s)$ denote the Dedekind zeta function of $F$.
If $\kappa_1$ and $\kappa_2$ are even, for any element $n\in\fo_F$, $n\gg 0$, we obtain
\begin{align}\label{eq:trace}
\tr\, \mathbb{T}_n|_{S_{\underline{\kappa}}(\Gamma)} + \mathscr{R}(\underline{\kappa},h_n)  = &\frac{1}{2} \zeta_F(-1) \, (\kappa_1-1)(\kappa_2-1) \, \overline{h_n}(1)   \\
&  + \sum_{\{{\overline{\gamma}}\}_{\overline{\rG}(F)}}\overline{\Phi}({\overline{\gamma}},\overline{h_n})\, U_{\kappa_1-2}( \cos(t_1))\, U_{\kappa_2-2}(\cos(t_2))  ,  \nonumber
\end{align}
where $\{{\overline{\gamma}}\}_{\overline{\rG}(F)}$ runs over all regular $\R$-elliptic $\overline{\rG}(F)$-conjugacy classes and $t_j$ $(j=1,2)$ is defined so that $\iota_j({\overline{\gamma}})$ is $\overline{\rG}(\R)$-conjugate to $\overline{\gamma_{t_j}}$. 
\end{theorem}
\begin{proof}
This follows from Lemmas \ref{lem:31} and \ref{lem:32} and \eqref{eq:JPhi}. 
\end{proof}

\subsection{Plancherel density theorem}

Assume that Conditions (I) and (II) hold. 
For each weight $\underline{\kappa}$, set 
$d_{\underline{\kappa}} \coloneqq \dim S_{\underline{\kappa}}(\Gamma)$, 
and let $\{\phi_{\underline{\kappa},j}\}_{1\le j\le d_{\underline{\kappa}}}$ 
be an orthonormal basis of $S_{\underline{\kappa}}(\Gamma)$ consisting of Hecke eigenforms. 
Recall that $\lambda_{\phi}(n)$ is the eigenvalue of a Hecke eigenform $\phi$ with respect to $\mathbb{T}_n'=N_{F/\Q}(n)^{-1/2}\mathbb{T}_n$. 
Then,
\begin{equation*}
\tr\, \mathbb{T}_n|_{S_{\underline{\kappa}}(\Gamma)}=N_{F/\Q}(n)^{1/2}\sum_{j=1}^{d_{\underline{\kappa}}}  \lambda_{\phi_{\underline{\kappa},j}}(n), \qquad n\in\fo_F,\quad n\gg 0.
\end{equation*}

Fix a prime ideal $(p)$ of $\fo_F$ with $p \gg 0$. 
Let $\delta_x$ denote the Dirac measure at $x$ on the interval $\Omega \coloneqq [-2,2]$, and $C^0(\Omega)$ the space of $\R$-valued continuous functions on $\Omega$. 
By Theorem \ref{thm:GRC}, we have $\lambda_{\phi_{\underline{\kappa},j}}(p)\in\Omega$, and then we define a measure $\mathscr{T}(p,\underline{\kappa})$ on $\Omega$ by
\[
\mathscr{T}(p,\underline{\kappa},f)\coloneqq \frac{1}{d_{\underline{\kappa}}}\sum_{j=1}^{d_{\underline{\kappa}}} \delta_{x(p,\underline{\kappa},j)}(f),\quad x(p,\underline{\kappa},j)\coloneqq \lambda_{\phi_{\underline{\kappa},j}}(p), \quad f\in C^0(\Omega).
\]
Note that $\delta_{x}(f)=f(x)$. 
The following theorem is analogous to \cite{Ser97}*{Th\'eor\`eme~1}.
\begin{theorem}\label{thm:PDT}
    Take an arbitrary sequence of weights $\{\underline{\ell}(r)\}_{r\in\N}$ such that $\underline{\ell}(r)=(\ell_{1,r},\ell_{2,r})$ and $\ell_{1,r}+\ell_{2,r}\to\infty$ as $r\to\infty$. 
    Then, for any $f\in C^0(\Omega)$, we have
    \[
    \lim_{r\to \inf} \mathscr{T}(p,\underline{\ell}(r),f)=\int_\Omega f(x)\rd \mu_p(x), \qquad \mu_p(x)\coloneqq \frac{q+1}{\pi} \frac{(1-x^2/4)^{\frac12} \rd x}{(q^{\frac{1}{2}}+q^{-\frac{1}{2}})^2-x^2} .
    \]
\end{theorem}
\begin{proof}
    This is a consequence of Theorem \ref{thm:traceformula} and the argument in \cite{Ser97}. For this reason, we omit the proof.
\end{proof}

\section{Elliptic terms}\label{sec:Elliptic}

Assume that the condition (I) holds. Let $\overline{\gamma}\in \overline{\rG}(F)$ be an $\R$-elliptic element. Fix $\gamma\in\rG(F)$ as its representative such that
\begin{align}
    t_0:=\mathrm{tr}\ \gamma\in\fo_F,\quad n_0:=\det\gamma\in\fo_F.\label{gamma}
\end{align}
Then the minimal polynomial of $\gamma$ over $F$ is equal to $X^2-t_0X+n_0$.

The aim of this section is to calculate the orbital integral $\overline{\Phi}(\overline{\gamma},\overline{h})$ given in \eqref{OrbitalInt}
where $h=h_{n}$ is defined in \S\ref{sec:Hecke operators}. 
We shall imitate the argument developed in \cite{KL06}*{\S 26}. 
To this end, we first compute the value of the orbital integral considered in \cite{KL06},
\[
\Phi(\gamma,\overline{h})\coloneqq \int_{\overline{\rG_\gamma(F)}\bs \overline{\rG}(\A_f)} \overline{h}(g^{-1}\gamma g)\, \rd g,
\]
and then derive $\overline{\Phi}(\overline{\gamma},\overline{h})$ by examining the difference between $\overline{\rG_\gamma(F)}$ and $\overline{\rG}_{\overline{\gamma}}(F)$.
Then we can view $\rG_{\gamma}(F)$ as the multiplicative group of the number field $E:=F(\sqrt{t_0^2-4n_0})$. In what follows, we identify $E$ with $\rG_{\gamma}(F)\cup\{0\}$ as sets. Since $\fo_F$ is a principal ideal domain due to the assumption (I), the ring of integers $\fo_E$ in $E$ is a free $\fo_F$-module of rank $2$. Hence, there exists an element $\epsilon\in E$ such that $\fo_E=\fo_F\oplus\epsilon\fo_F$.
\begin{lemma}\label{multiplicationLem}
After replacing $\gamma$ by a possible $\rG(F)$-conjugate, it holds that
\begin{align}
\epsilon[1\ \epsilon]=[1\ \epsilon]\epsilon\label{product in E}
\end{align}
as row vectors. Here, the left-hand side is the multiplication in $E$ (i.e. this equals $[\epsilon\ \epsilon^2]$), the right-hand side is the multiplication of the row vector $[1\ \epsilon]$ and the matrix $\epsilon$.
\end{lemma}

\begin{proof}
By considering a certain $\rG(F)$-conjugate of $\epsilon$, we may assume  
$\epsilon=\left[\begin{smallmatrix} {0} & {b} \\ {1} & {a} \end{smallmatrix}\right]$
for some $a,b\in F$. This is possible by a direct computation. Then we have
$$
[\epsilon, \epsilon^2]=[\epsilon,a\epsilon+b]=[1\ \epsilon]\left[\begin{smallmatrix} {0} & {b} \\ {1} & {a} \end{smallmatrix}\right].
$$
Hence the statement holds.
\end{proof}
Since the integral $\Phi(\gamma,h)$  is invariant under replacing $\gamma$ by a $\rG(F)$-conjugate, this operation causes no issue. In the following, we always assume that \eqref{product in E} holds.

Let $\fL$ denote the set of all free $\fo_F$-modules $L\subset F^2$ of rank 2. We view elements in $L$ as column vectors. Then the group $\rG(F)$ acts on the set $\fL$ from the left via the standard matrix multiplication. Similarly, for finite place $v$, let $\fL_v$ be the set of all $\fo_{F_v}$-modules of rank $2$. Set
$$
\fL_{\A_f}:=
\left\{
(L_v)_{v<\infty}\in\prod_{v<\infty}\fL_v
\mid
L_v=\left[\begin{smallmatrix} {\fo_{F_v}} \\ {\fo_{F_v}} \end{smallmatrix}\right]\text{ for almost all }v
\right\}.
$$

The next proposition is the local-global principle of lattices (cf. \cite{KL06}*{Theorem 26.15}).
\begin{proposition}\label{LGPlattice}
The correspondence
$$
\Psi:
\fL_{\A_f}\owns (L_v)_{v<\infty}
\mapsto
\cap_{v<\infty}(L_v\cap\left[\begin{smallmatrix} {F} \\ {F} \end{smallmatrix}\right])\in \fL
$$
is well-defined and bijective.
\end{proposition}
\begin{proof}
Set $L=\Psi((L_v)_{v<\infty})$. First, we shall prove the well-definedness of the map, i.e. $L$ is in $\fL$. For a finite place $v$, we fix a prime element $\varpi_v\in\fo_F$ corresponding to $v$. Then, for each $v$ satisfying $L_v\ne\left[\begin{smallmatrix} {\fo_{F_v}} \\ {\fo_{F_v}} \end{smallmatrix}\right]$, there exists an integer $m_v\in\Z_{>0}$ such that  
$
\varpi^{m_v}\left[\begin{smallmatrix} {\fo_{F_v}} \\ {\fo_{F_v}} \end{smallmatrix}\right]
\subset L_v
\subset \varpi^{-m_v}\left[\begin{smallmatrix} {\fo_{F_v}} \\ {\fo_{F_v}} \end{smallmatrix}\right]
$.
We put $m_v=0$ when $L_p=\left[\begin{smallmatrix} {\fo_{F_v}} \\ {\fo_{F_v}} \end{smallmatrix}\right]$ and  $x=\prod_{v<\infty}\varpi^{m_v}$. Then  we have
$
x\left[\begin{smallmatrix} {\fo_{F}} \\ {\fo_{F}} \end{smallmatrix}\right]
\subset L
\subset x^{-1}\left[\begin{smallmatrix} {\fo_{F}} \\ {\fo_{F}} \end{smallmatrix}\right]
$. This proves $L\in\fL$. 

Second, we check the surjectivity. For each $L\in\fL$, fix an $\fo_F$-basis 
$\{\mathbf{x},\mathbf{y}\}$ of $L$ and set $L_v:=L\otimes_{\fo_F}\fo_{F_v}=\mathbf{x}\fo_{F_v}\oplus\mathbf{y}\fo_{F_v}$. Since $\mathbf{x}\fo_{F_v}\oplus\mathbf{y}\fo_{F_v}=\fo_{F_v}^2$ for almost all finite places $v$, the family $(L_v)_{v<\infty}$ is in $\fL_{\A_f}$. Then our claim is $\Psi((L_v)_{v<\infty})=L$. The inclusion $\supset$ is obvious. On the other hand, let $a\mathbf{x}+b\mathbf{y}\in \Psi((L_v)_{v<\infty})$ ($a,b\in F$) be an arbitrary element. Because  $\{\mathbf{x},\mathbf{y}\}$ is also an $F_v$-basis of $F_v^2$, $a\mathbf{x}+b\mathbf{y}\in\cap_{v<\infty}L_v$ implies that $x,y\in F\cap(\cap_{v<\infty}\fo_{F_v})=\fo_F$. Hence the converse inclusion $\subset$ holds.

Finally, it remains to prove the injectivity. For $(L_v)_{v<\infty}\in\fL_{\A_f}$, we set $M=\Psi((L_v)_{v<\infty})$. It suffices to show that $M\otimes_{\fo_F}\fo_{F_v}=L_v$ for all finite places $v$. Let $\mathbf{L}=\prod_{v<\infty}L_v\subset\A_f^2$, then, by regarding $M$ as a subset of $\A_f^2$ via diagonal embedding, it is clear that $M=\mathbf{L}\cap F^2$. By the strong approximation \eqref{eq:I1}, $F^2$ is dense in $\A_f^2$. Because $\mathbf{L}$ is an open subset of $\A_f^2$, it follows that $M$ is dense in $\mathbf{L}$, and hence $M$ is dense in $L_v$ for all finite places $v$. Since $M\otimes_{\fo_F}\fo_{F_v}$ is a closed subset of $L_v$ containing $M$, it must coincide with $L_v$. This completes the proof.

\end{proof}
It is easy to check that there exists an isomorphism
$$
\lambda:\rG(\A_f)/K_f\to\fL_{\A_f}
$$
defined by 
$
\lambda((g_v)_{v<\infty})
=\left(g_v\left[\begin{smallmatrix} {\fo_{F_v}} \\ {\fo_{F_v}} \end{smallmatrix}\right]\right)_{v<\infty}
$.
By combining this and Proposition \ref{LGPlattice}, we obtain the following corollary.
\begin{corollary}\label{Cor.1}
The map
$$
\Psi\circ\lambda:\rG(\A_f)/K_f\to\fL
$$
is a bijection. 
\end{corollary}

Ultimately, the orbital integral we seek can be expressed in terms of the class number of a certain order. We recall some definitions related to orders.
An order $\fo\subset E$ is a subring of $\fo_E$ that is a free $\Z$-module satisfying $\Q\fo=\fo_E$. We note that an order containing $\fo_F$ naturally becomes a free $\fo_F$-module. Let $\fL_E$ denotes the set of all free $\fo_F$-modules $\cL\subset E$ of rank 2. Then there exists a one-to-one correspondence:
\begin{align}
\fL\owns L\stackrel{\sim}{\mapsto} \cL\in\fL_E\label{Corr.latices}
\end{align}
where $\cL:=\{x+y\epsilon\mid \left[\begin{smallmatrix} {x} \\ {y} \end{smallmatrix}\right]\in L\}$. We note that this map is independent of the choice of $\fo_F$-basis of $\fo_E$. As in \eqref{Corr.latices}, in what follows, we write elements of 
$\fL$ in italics and elements of $\fL_E$ in calligraphic script. For $\cL\in\fL_E$, the order $\fo_{\cL}$ corresponding to $\cL$ is defined by 
$
\fo_{\cL}\coloneqq\{ \mu\in E \mid \mu\cL\subset\cL\}
$.
For an order $\fo$ containing $\fo_F$ and a lattice $\cL$, we say that $\cL$ is $\fo$-lattice if $\fo\cL_E=\cL_E$ holds under the standard multiplication in $E$. For a lattice $\cL\in\fL_E$, let us denote an order $\{\mu\in E\mid \mu\cL\subset \cL\}$ by $\fo_{\cL}$ and call this the order corresponding to $\cL$. For $\fo$-lattice $\cL$, we say that $\cL$ is proper if $\fo_{\cL}=\fo$.

For an $\fo$-lattice $\cL$, we say that $\cL$ is invertible if there exists an $\fo$-lattice $\cL^{-1}\in\fL_E$ such that $\cL\cL^{-1}=\fo$, and $\cL$ is principal if there exists $\alpha\in E^\times$ such that $\cL=\alpha\fo$. We denote the group of all invertible $\fo$-lattices by $\cI(\fo)$, and all principal $\fo$-lattices by $\cP(\fo)$. 
\begin{proposition}\label{ProperInvertible}
Let $\fo\subset E$ be an order containing $\fo_F$ and $\cL\in\fL_E$ be an $\fo$-lattice. Then $\cL$ is invertible if and only if $\cL$ is proper.
\end{proposition}
To prove the proposition, we need the following lemma
\begin{lemma}\label{Corr. lattice}
Let $\cL\in\fL_E$ be a lattice of the form $\cL=\fo_F\oplus\beta\fo_{F}$ with $\beta\in E$ whose minimal polynomial over $\fo_F$ is $aX^2+bX+c$ with relatively prime elements $a,b,c\in\fo_F$. Then $\fo_{\cL}=\fo_F\oplus a\beta\fo_F$.
\end{lemma}

\begin{proof}
This can be easily shown in the same way as \cite{KL06}*{Lemma 26.10}
\end{proof}

\begin{proof}[Proof of Proposition \ref{ProperInvertible}]
Suppose that $\cL$ is invertible, then the inclusion $\fo\subset\fo_{\cL}$ is clear. For any $\lambda\in\fo_{\cL}$, we have $\lambda\in\lambda\fo=\lambda\cL\cL^{-1}\subset\cL\cL^{-1}=\fo$. Thus, we have $\fo\supset\fo_{\cL}$, and hence $\fo=\fo_{\cL}$.
Conversely, suppose that $\cL$ is proper $\fo$-lattice. By multiplying $\cL$ by a certain element in $E^\times$, We may assume that $\cL=\fo_F\oplus\beta\fo_F$ for some $\beta\in F^2$. Let $aX^2+bx+c$ be a minimal polynomial of $\beta$ over $\fo_F$ with relatively prime elements $a,b,c\in\fo_F$. Define the conjugate lattice $\overline{\cL}$ of $\cL$ by 
$$
\overline{\cL}:=\{\left[\begin{smallmatrix} {\overline{x}} \\ {\overline{y}} \end{smallmatrix}\right]\mid \left[\begin{smallmatrix} {x} \\ {y} \end{smallmatrix}\right]\in \cL\}
$$
where $\overline{x}$ is the unique non-trivial Galois conjugate of $x$ on $E/F$. Then we get
\begin{align*}
\cL\overline{\cL}=
\fo_F+\beta\fo_F+\overline{\beta}\fo_F&+\beta\overline{\beta}\fo_F=
\fo_F+\beta\fo_F+\left(-\frac{b}{a}-\beta\right)\fo_F+\frac{c}{a}\fo_F\\
&=a^{-1}(a\fo_F+b\fo_F+c\fo_F)+\beta\fo_F=a^{-1}(\fo_F+a\beta\fo_F)=a^{-1}\fo_{\cL}=a^{-1}\fo\\
\end{align*}
by Lemma \ref{Corr. lattice}. Therefore, we have $\cL^{-1}=a\overline{\cL}$, and hence $\cL$ is invertible.
\end{proof}

\begin{proposition}\label{KL prop26.23}
    There exist one-to-one correspondences between the following five sets:
    \begin{itemize}
        \item[(i)] $\overline{\rG_\gamma(F)}\bs\overline{\rG}(\A_f)/\overline{K_f}$
        \item[(ii)] $\rG_\gamma(F)\bs\rG(\A_f)/K_f$
        \item[(iii)] $\rG_\gamma(F)\bs\fL$
        \item[(iv)] $E^\times\bs\fL_E$
        \item[(v)] $\cup_{\substack{\fo\subset E:\text{order}\\ \fo_F\subset\fo}}\left(\cP(\fo)\bs\cI(\fo)\right)$
    \end{itemize}
\end{proposition}
\begin{proof}
The correspondence (i)$\leftrightarrow$(ii) follows from the isomorphism
$$\overline{\rG_\gamma(F)}\bs\overline{\rG}(\A_f)/\overline{K_f}\cong\rG_\gamma(F)\bs\rG(\A_f)/K_f
$$
due to  $Z(F)\subset\rG_{\gamma}(F)$, $Z(\widehat{\fo_F})\subset K_f$, and the strong approximation \eqref{eq:I1}.
The correspondences (ii)$\leftrightarrow$(iii) and (iii)$\leftrightarrow$(iv) are direct consequences from Corollary \ref{Cor.1}, Lemma \ref{multiplicationLem}, and \eqref{Corr.latices}. Finally, we see the identification (iv) with (v). Indeed, 
$$
E^\times\bs\fL_E=\cup_{\substack{\fo\subset E:\text{order}\\ \fo_F\subset\fo}}\left(E^\times\bs\{\cL\in\fL_E\mid \fo_{\cL}=\fo\}\right)=\cup_{\substack{\fo\subset E:\text{order}\\ \fo_F\subset\fo}}\left(\cP(\fo)\bs\cI(\fo)\right)
$$
by Proposition \ref{ProperInvertible}. This completes the proof.
\end{proof}

\begin{lemma}\label{Support}
For $g\in\rG(\A_f)$, let $L=\Psi\circ\lambda(L)$ be the corresponding lattices in $\fL$. Then $h_{n}(g^{-1}\gamma g)\ne0$ if and only if $n/n_0\in\fo_F^\times$ and $\gamma L\subset L$.  
\end{lemma}
\begin{proof}
For an element $x\in\rG(\A_f)$, it is easy to check that $x$ is in the support of $h$ if and only if $\det x\in\alpha U_f$ and 
$
x\left[\begin{smallmatrix} {\widehat{\fo_F}} \\ {\widehat{\fo_F}} \end{smallmatrix}\right]
\subset
\left[\begin{smallmatrix} {\widehat{\fo_F}} \\ {\widehat{\fo_F}} \end{smallmatrix}\right]
$. Hence, the condition $h_{\alpha}(g^{-1}\gamma g)\ne 0$ is equivalent to that $g$ satisfies
$$
\det(g^{-1}\gamma g)=\det \gamma=n_0\in\alpha U_f
$$
and
$$
\gamma g\left[\begin{smallmatrix} {\fo_{F_v}} \\ {\fo_{F_v}} \end{smallmatrix}\right]
\subset
g\left[\begin{smallmatrix} {\fo_{F_v}} \\ {\fo_{F_v}} \end{smallmatrix}\right]
\text{ for all finite places $v$}.
$$
By the local-global principle of lattices (Proposition \ref{LGPlattice}), the second condition is equivalent to $\gamma L\subset L$. Hence the Lemma follows.
\end{proof}
\begin{lemma}\label{KL Lem26.25}
For $g\in\rG(\A_f)$, let $\cL\in\fL_E$ be the lattice corresponding to $g$ via the composition of $\Psi\circ\lambda$ and \eqref{Corr.latices}. Then 
$$
\#(g\overline{K_f}g^{-1}\cap \overline{\rG_{\gamma}(F)})
=[\fo_{\cL}^\times:\fo_F^\times].
$$
\end{lemma}
\begin{proof} 
Since $\overline{\fo_{\cL}^\times}\cong \fo_{\cL}^\times/(Z(\A_f)\cap\fo_{\cL}^\times)=\fo_{\cL}^\times/\fo_F^\times$, it suffices to prove $g\overline{K_f}g^{-1}\cap \overline{\rG_{\gamma}(F)}=\overline{\fo_{\cL}^\times}$.  Let $\overline{x}\in g\overline{K_f}g^{-1}\cap \overline{\rG_{\gamma}(F)}$ be an element represented by $x\in\rG(\A_f)$. Then there exist $z_1,z_2\in Z(\A_f)$, and $\delta\in\rG_{\gamma}(F)$ such that $z_1^{-1}x=z_2\delta \in gK_fg^{-1}$. This yields that $z_2\delta g\left[\begin{smallmatrix} {\fo_{F_v}} \\ {\fo_{F_v}} \end{smallmatrix}\right]=g\left[\begin{smallmatrix} {\fo_{F_v}} \\ {\fo_{F_v}} \end{smallmatrix}\right]$ for all finite places $v$. The local-global principle of lattices and Lemma \ref{multiplicationLem}, imply $z_2\delta\cL=\cL$. This is equivalent to $z_1^{-1} x=z_2\delta\in\fo_{\cL}^\times$. Hence $\overline{x}\in\overline{\fo_{\cL}^\times}$. Therefore we have 
$g\overline{K_f}g^{-1}\cap \overline{\rG_{\gamma}(F)}\subset \overline{\fo_{\cL}^\times}$. The converse inclusion follows by reversing the above process. 
\end{proof}

\begin{proposition}\label{prop:KL}
Let $\gamma\in \rG(F)$ be an elliptic element satisfying the condition \eqref{gamma} and $n/n_0\in\fo_F^\times$. Then we have
\begin{align*}
\Phi(\overline{\gamma},\overline{h})=\sum_{\substack{\fo\subset E:\text{order}\\ \gamma\in\fo, \fo_F\subset\fo}}
\frac{h(\fo)}{[\fo^\times:\fo_F^\times]}=h_{E}[\fo_{E}^\times:\fo_F^\times]^{-1}\sum_{\substack{\fo\subset E:\text{order}\\ \gamma\in\fo, \fo_F\subset\fo}}N_{F/\Q}(\ff_{\fo})
\prod_{\fp|\ff_\fo}(1-(\tfrac{t^2-4n}{\fp})N_{F/\Q}(\fp)^{-1})
\end{align*}
where $h(\fo):=\#(\cP(\fo)\bs\cI(\fo))$ and $\ff_\fo$ is the conductor of $\fo$, which is defined to be an ideal $\fo_F$ satisfying $\fo=\fo_F\oplus\ff_{\fo}\epsilon$.
\end{proposition}

\begin{proof}
Because the function 
$\overline{\rG}(\A_f)\owns g\mapsto h(g^{-1}\gamma g)$ is right $\overline{K_f}$-invariant, we have
$$
\Phi(\overline{\gamma},\overline{h})=\sum_{\overline{\rG_{\gamma}(F)}\bs\overline{\rG}(\A_f)/\overline{K_f}}
\mathrm{vol}_{\overline{\rG_{\gamma}(F)}\bs\overline{\rG}(\A_f)}(\overline{\rG_{\gamma}(F)} g\overline{K_f})\cdot \int_{Z(\A_f)}h(zg^{-1}\gamma g)\rd z.
$$
It is obvious that the last $Z(\A_f)$-integral equals $h(g^{-1}\gamma g)$ in view of the support condition. The corresponding volume factor can be calculated by the right-$\rG(\A_f)$ invariance of the measure and by using Lemma \ref{KL Lem26.25} as follows:
\begin{align*}
\mathrm{vol_{\overline{\rG_{\gamma}(F)}\bs\overline{\rG}(\A_f)}}(\overline{\rG_{\gamma}(F)} g\overline{K_f})&=
\mathrm{vol_{\overline{\rG_{\gamma}(F)}\bs\overline{\rG}(\A_f)}}(\overline{\rG_{\gamma}(F)} g\overline{K_f}g^{-1})\\
&=\frac{\mathrm{vol}_{\overline{\rG}(\A_f)}(g\overline{K_f}g^{-1})}{\#( g\overline{K_f}g^{-1}\cap\overline{\rG_{\gamma}(F)})}
=(\fo^\times:\fo_F^\times)^{-1}
\end{align*}
where $\fo=\fo_{\cL}$ and $\cL$ is a lattice in $\cL_E$ corresponding to $g$. By Lemma \ref{Support}, $h(g^{-1}\gamma g)\ne 0$ if and only if $\gamma\in\fo_{\cL}$. This and the correspondence (ii)$\leftrightarrow$(v) in Proposition \ref{KL prop26.23} reveal the first equation. The second equation follows from \cite{Neu99}*{Theorem 12.12}.
\end{proof}

\begin{proposition}\label{prop:KL2}
Let $\gamma\in\rG(F)$ be as in Proposition \ref{prop:KL}. Then we have $\Phi(\gamma,\overline{h})=0$ unless $t,n\in\fo_{F}$, equivalently unless $\gamma\in\fo_E$. In this case, it holds that
$$
\Phi(\gamma,\overline{h})=H_F(4n-t^2).
$$
Here $H_F(\xi)$ is the generalized $\xi$th Hurwitz class number defined by \eqref{GHCN}.

\end{proposition}

\begin{proof}
The first assertion follows immediately from Proposition \ref{prop:KL}. Indeed, if $\gamma$ is contained some lattice $\fo$, then $\gamma\in\fo_E$ since $\fo\subset\fo_E$ for any order $\fo\subset E$. When $\gamma\in\fo_E$, we define an order $\fo_{\gamma}$ attached to $\gamma$ by $\fo_{\gamma}=\fo_F\oplus\gamma\fo_F$. We put $\gamma=\alpha+\beta\epsilon$ $(\alpha,\beta\in\fo_F)$. Then we obtain $\ff_{\fo_\gamma}=(\beta)$ and
$\fD_{t^2-4n}=((\epsilon-\overline{\epsilon})^2)=(\beta^{-2}(t^2-4n))$. Hence we get $\fF_{t^2-4n}=\ff_{\fo_{\gamma}}$. Therefore, for any order $\fo\subset E$, $\fo_{\gamma}\subset\fo$ if and only if $\ff_{\fo}|\fF_{t^2-4n}$. This condition, Proposition \ref{prop:KL}, and \eqref{GHCN} reveal the desired formula.
\end{proof}

\begin{lemma}\label{lem:KL3}
Let $\gamma\in\overline{\rG}(F)$. Then,
$$
\overline{\Phi}(\overline{\gamma},\overline{h})=\Phi(\gamma,\overline{h}) \times \begin{cases}
    1 & \text{if $\tr(\gamma)\neq 0$,} \\
    2 & \text{if $\tr(\gamma)=0$.}
\end{cases}
$$
\end{lemma}
\begin{proof}
This fact follows from the difference between $\overline{\rG_\gamma(F)}$ and $\overline{\rG}_{\overline{\gamma}}(F)$, see \cite{KL06}*{p.295}.    
\end{proof}
\begin{remark}\label{rem:conjugacyclass}
Fix $n$, and let $\gamma_j \in \rG(F)$ $(j=1,2)$ be as in Proposition \ref{prop:KL}, 
so that the minimal polynomial of $\gamma_j$ over $F$ is given by 
$X^2 - t_j X + n$, with $4n - t_j \gg 0$.
By the Skolem--Noether theorem, note that $\gamma_1$ and $\gamma_2$ are $G(F)$-conjugate if and only if $t_1 = t_2$.
\end{remark}

\section{Proof of the resolvent trace formula}
For each Hecke eigenform $\phi\in S_{\underline{\kappa}}(\Gamma_1)$ and a totally positive prime element $p\in\fo_F$, we may write $\lambda_\phi(p)=\alpha_{\phi}(p)+\alpha_{\phi}(p)^{-1}$
where $(\alpha_{\phi}(p),\alpha_\phi(p)^{-1})$ is the Satake parameter of $\phi$ at $p$, which satisfies $|\alpha_{\phi}(p)|=1$ by Theorem \ref{thm:GRC}. Then one has $(\mathbb{T}_{p}'-(X+X^{-1}))^{-1}\phi=(\alpha_{\phi}(p)+\alpha_{\phi}(p)^{-1}-(X+X^{-1}))^{-1}\phi$ for $|X|<1$.
\begin{proof}[Proof of Theorem \ref{RTF}]
Taking into account the basis $\{\phi_{\underline{\kappa},j}\}_{1\le j\le d_{\underline{\kappa}}}$ of $S_{\underline{\kappa}}(\Gamma_1)$ and the following identity:
\begin{align}
\frac{1}{\alpha+\alpha^{-1}-(X+X^{-1})}=-\sum_{\nu=0}^{\infty}U_{\nu}\left(\frac{\alpha+\alpha^{-1}}{2}\right)X^{\nu+1},\quad |X|<|\alpha|,\label{Id5}
\end{align}
we find that the desired trace is equal to
\begin{align*}
    -\sum_{\nu=0}^{\infty}\tr \mathbb{T}_{p^{\nu}}\cdot X^{\nu+1}.
\end{align*}
Here, we use the well-known recursions:
\begin{align*}
\mathbb{T}_{p^{\nu+1}}'=\mathbb{T}_{p}'\mathbb{T}_{p^{\nu}}'-\mathbb{T}_{p^{\nu-1}}',
\quad U_{\nu+1}(x)=U_{1}(x)U_{\nu}(x)-U_{\nu-1}(x),\quad \nu\ge1,
\end{align*}
in the above argument. 
Recall that $q= N_{F/\Q}(p)$. 
An explicit formula for $\tr \mathbb{T}_{p^{\nu}}'=q^{-\frac{\nu}{2}}\tr \mathbb{T}_{p^{\nu}}$ is already given in Theorem \ref{thm:ES}, and all terms except for the second are expressed in a simple form. By computing a simple geometric series for the first term and applying the identity \eqref{Id5} again for the third term, the proof is completed. 
\end{proof}

\begin{proof}[Proof of Corollary \ref{Optimal}]
By a short calculation, Theorem \ref{RTF} reveals
\begin{align}
    \sum_{\nu=0}^{\infty}A_{\underline{\kappa}}(p^{\nu})(q^{-\frac{1}{2}}X)^{\nu}=&
    \sum_{j=1}^{d_{\underline{\kappa}}}\frac{2}{(X-\alpha_{\phi_{\underline{\kappa},j}}(p))(X-\alpha_{\phi_{\underline{\kappa},j}}(p)^{-1})}\label{RTF2}\\
    &+\frac{1}{24}B_{2,\chi_D}(\kappa_1-1)(\kappa_2-1)\frac{1}{1-q^{-1}X^2}\notag\\
    &+\begin{cases}
    \frac{2}{(1-q^{\frac{1}{2}}X)(1-q^{-\frac{1}{2}}X)}\quad &if\ \kappa_1=\kappa_2=2,\\
    0 & otherwise,
    \end{cases}\notag
\end{align}
if $|X|$ is sufficiently small. We may suppose $f(x,y)=U_{l_1}(x)U_{l_2}(y)$ with $(l_1,l_2)\in(2\Z_{\ge 0})^{2}$. 
By applying the formula \eqref{RTF2} for the case $\underline{\kappa}=(l_1,l_2)$ and using the result of Theorem \ref{thm:GRC}, it turns out that the radius of convergence of the power series on the left-hand side is $q^{-\frac{1}{2}}$ or $1$ according as 
$(l_1,l_2)$ equals $(0,0)$ or not. Hence the statement holds by Cauchy's estimate.
\end{proof}

\begin{proof}[Proof of Theorem \ref{Equid}]
To confirm the assertion, it suffices to show the following:
\begin{align}
\lim_{\nu\to\infty}\langle \mu_{p,\nu},U_{l_1}(x)U_{l_2}(y)\rangle=\delta_{0,l_1}\delta_{0,l_2},\quad (l_1,l_2)\in(2\Z_{\ge 0})^2.\label{Asymp}
\end{align}
When $(l_1,l_2)\ne(0,0)$, we verify $\langle \mu_{p,\nu},U_{l_1}(x)U_{l_2}(y)\rangle=o(q^{-\frac{\nu}{2}})$ $(\nu\to\infty)$ from Corollary \ref{Optimal} and hence it tends to $0$ as $\nu\to \infty$. Next, we assume $(l_1,l_2)=(0,0)$. 
Since the radius of convergence of the left-hand side on \eqref{RTF2} is $q^{-\frac{1}{2}}$, and the first two terms on the right-hand side are holomorphic on $|X|<1$, which clearly contains $|X|<q^{-\frac{1}{2}}$, we have the estimate
\begin{align}
    q^{-\frac{\nu}{2}}A_{(2,2)}(p^{\nu})
    \sim
    2\frac{q^{\frac{\nu+1}{2}}-q^{-\frac{\nu+1}{2}}}{q^{\frac{1}{2}}-q^{-\frac{1}{2}}}\quad (\nu\to\infty)\label{Asymp2}
\end{align}
by comparing the coefficients of $X^\nu$. $ $Here, we use the fact that the right-hand side coincides with the coefficient of $X^\nu$ in $\frac{2}{(1-q^{\frac{1}{2}}X)(1-q^{-\frac{1}{2}}X)}$. We deduce the estimate \eqref{Asymp} for $(l_1,l_2)=(0,0)$ from \eqref{Asymp2}.
\end{proof}

\section{On algorithms for class numbers and traces of Hecke operators}\label{sec:algorithm}
In this section we describe some details on numerically implementing Theorem \ref{thm:ES}. To compute all traces $T_{n,F}$ where $N_{F/\Q}(n) < A$, we first need to compute all relevant generalized Hurwitz class numbers of the form $H_F(4n-t^2)$, where $4n-t^2$ is totally positive. To compute these, we directly use equation \eqref{eq:HF3}, noting that, for each increase of $n$, we roughly need another $\sqrt{N_{F/\Q}(n)}$ class numbers. Once we have the relevant list of class numbers, computing the geometric side of Theorem \ref{thm:ES} is rather straightforward, since it is essentially a weighted sum of generalized Hurwitz class numbers, hence the main bulk of the computation of the traces is in computing the class group structure of $F(\sqrt{4n-t^2})$. 

For an implementation, we consider the fields $F = \Q(\sqrt{5})$ and $F = \Q(\sqrt{29})$. For both fields, we computed the trace for all totally positive prime element $p$ with norm $N_{F/\Q}(p) \leq 2\times10^6$, requiring $3,410,968$ and $3,045,914$ Hurwitz class numbers for each field respectively. As an initial test of correctness of the Hurwitz class number computations, we consider the weight $(2,2)$, as this has the simplest formula. For $\Q(\sqrt{5})$, the space of Hilbert modular forms of weight $(2,2)$ is empty, thus we should expect all our traces to be $0$. We verified this to be true for all totally positive prime elements $p$ with norm $N_{F/\Q}(p) \leq 2\times10^6$. For $\Q(\sqrt{29})$, the space of Hilbert modular forms of weight $(2,2)$ is $1$-dimensional and the one form is associated to the elliptic curve \cite{lmfdb:2.2.29.1-1.1-a-curve} by modularity. In this case, we compared the trace values against computing points on the elliptic curve, which is substantially faster, taking only a few seconds to compute all the same data. Again, we get full agreement for all totally positive prime elements $p$ with norm $N_{F/\Q}(p) \leq 2\times10^6$. Furthermore, we see that the data agrees with \cite{lmfdb:2.2.29.1-1.1-a} for all prime ideals of norm up to $9967$. 

For $F = \Q(\sqrt{5})$, we use the Hurwitz class number data to test two pairs of weights: parallel weight $(8,8)$ and non-parallel weight $(4,8)$ (which will also compute its conjugate $(8,4)$). Both of these spaces are $1$-dimensional, allowing us to compute the Hecke eigenvalues directly. With this, we computed all $\lambda(p)$ with $N_{F/\Q}(p) \leq 2\times10^6$, giving $148,931$ $\lambda(p)$ values. In this case, we needed the same $3,410,968$ generalized Hurwitz class numbers. Using this data, we created the Sato--Tate plots in Figures \ref{fig:SatoTatePlot_5_88} and \ref{fig:SatoTatePlot_5_48}, which is a theorem in this case due to \cite{SatoTate}, to test correctness of the size of the Hecke eigenvalues.

For $\Q(\sqrt{29})$, we also test the non-parallel weight $(2,6)$, which is also $1$-dimensional. Here, we computed all $\lambda(p)$ with $N_{F/\Q}(p) \leq 2\times10^6$, giving $148,837$ $\lambda(p)$ values, using the $3,045,914$ Hurwitz class numbers computed. Similarly, we generated Sato--Tate plots for the $\Q(\sqrt{29})$ data in Figures \ref{fig:SatoTatePlot_29_22} and \ref{fig:SatoTatePlot_29_26}.

We remark that we normalize the trace formula values so that they take values in $[-2,2]$ by considering the normalization of the Hecke operator $\bT_n'= n^{-\frac{\kappa_1-1}{2}}\sigma(n)^{-\frac{\kappa_2-1}{2}} T_{n,F}$ given in Theorem~\ref{thm:GRC}. In the Sato--Tate figures below, the plots on the left are showing the distribution of the traces of the Hecke operators $\mathbb{T}'_{p}$ over totally positive prime elements $p$, in comparison to the semi-circle distribution. The plots on the right are showing the distribution on the angles $\theta_{p} = \arccos(\lambda(p)/2)$. 

This code was implemented in \texttt{PARI} \cite{PARI2}. Since we wanted some numerical validity to these tests, we computed the class number unconditionally using \texttt{PARI}, which is a $O(\Delta_{F(\sqrt{n})}^{1/2})$ algorithm (where $\Delta_{F(\sqrt{n})}$ is the discriminant of the number field extension $F(\sqrt{n})$), which is reasonable for the small discriminants we are working with. The class number computations took roughly $3$ hours for $\Q(\sqrt{5})$ and roughly $7$ and a half hours for $\Q(\sqrt{29})$ on $16$ cores (42.2 and 119.1 core hours respectively). The timings for the trace formula values are below:
\begin{table}[ht]
    \centering
    \begin{tabular}{c|c|c|c}
        Weight&D&Real time taken&Core hours\\
        \hline
        (2,2)&5&31 m 44 s&8.5\\
        (4,8)&5&36 m 9 s&9.6\\
        (8,8)&5&41 m 35 s&11.1\\
        (2,2)&29&10 m 48 s&2.9\\
        (2,6)&29&13 m 37 s&3.6
    \end{tabular}
    \caption{All the real timings in this table were run on 16 cores. The $D=29$ terms are faster as the number of terms in the elliptic sums grows proportional to $N_{F/\Q}(p)/\sqrt{D}$.}
    \label{tab:tracetimetable}
\end{table}
For future work, we aim to produce a larger list of class numbers by generalizing the work of Jacobson, Ramachandran, and Williams \cite{JRW06} to remove the GRH condition to generalized Hurwitz class numbers, using the trace formula. This would then allow for a larger computation in the future of the Hecke eigenvalues of Hilbert modular forms. The code used in this paper is available at \cite{code2026}.

\begin{figure}[H]
    \centering
    \includegraphics[width=\linewidth]{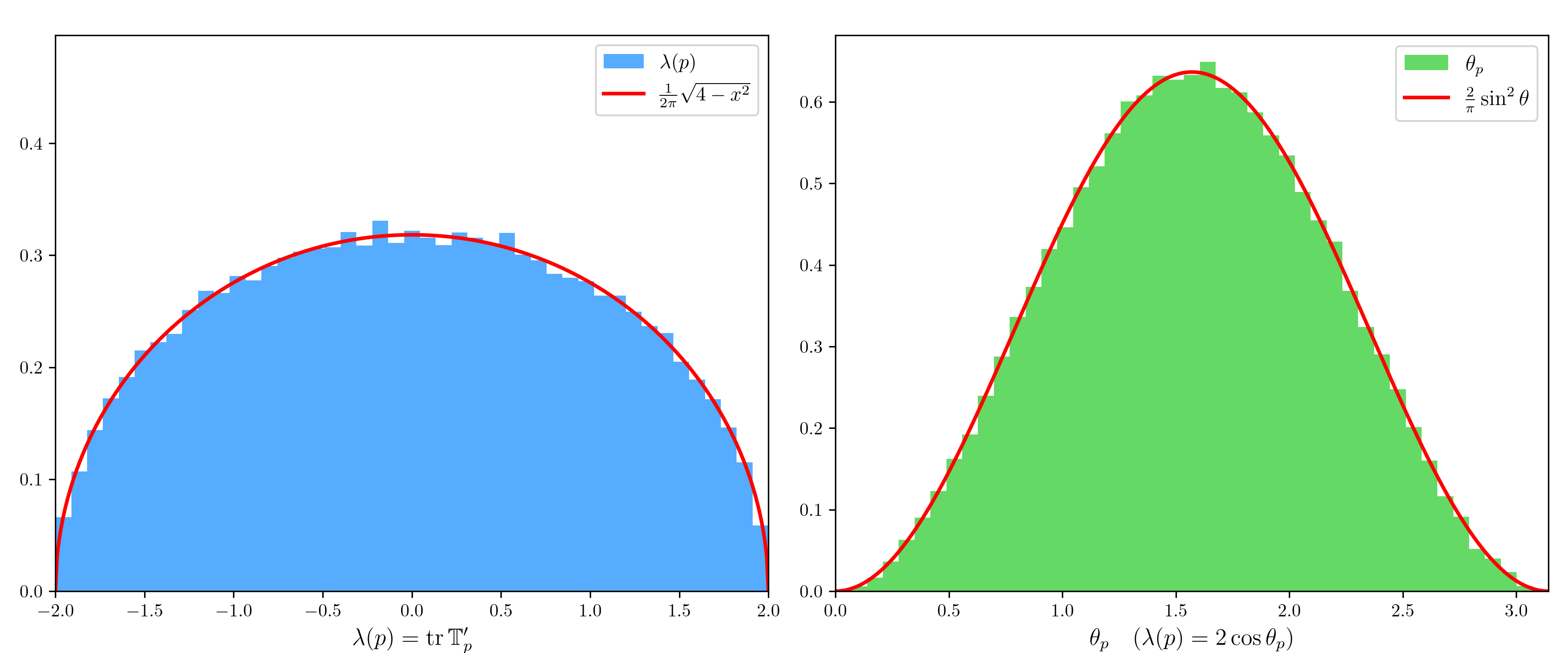}
    \caption{Sato--Tate plot for $\Q(\sqrt{29})$ and weight $(2,2)$. This has $148,837$ totally positive prime elements in $45$ bins.}
    \label{fig:SatoTatePlot_29_22}
\end{figure}
\begin{figure}[H]
    \centering
    \includegraphics[width=\linewidth]{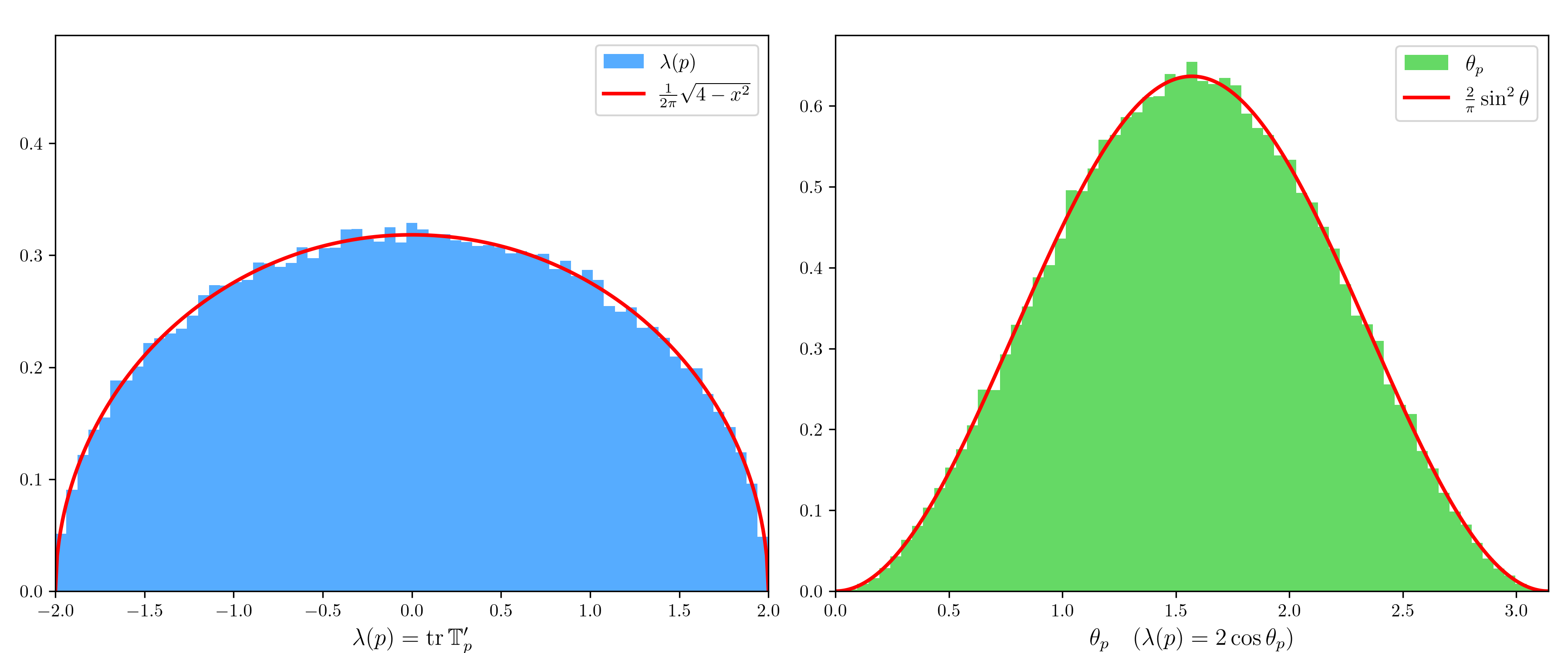}
    \caption{Sato--Tate plot for $\Q(\sqrt{29})$ and weight $(2,6)$. This has $148,837$ totally positive prime elements (giving here $297,674$ data points as we can consider both Galois embeddings) in $65$ bins.}
    \label{fig:SatoTatePlot_29_26}
\end{figure}
\begin{figure}[H]
    \centering
    \includegraphics[width=\linewidth]{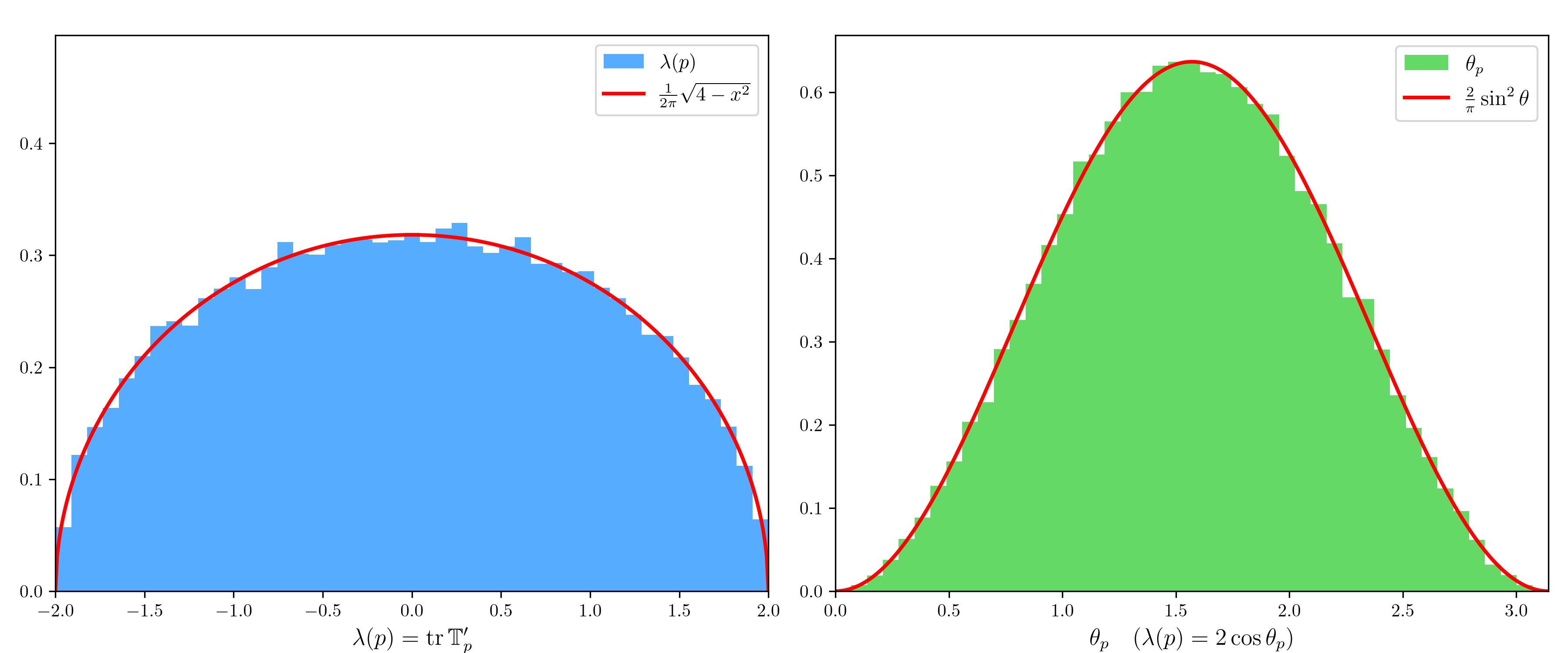}
    \caption{Sato--Tate plot for $\Q(\sqrt{5})$ and weight $(8,8)$. This has $148,931$ totally positive prime elements in $45$ bins.}
    \label{fig:SatoTatePlot_5_88}
\end{figure}
\begin{figure}[H]
    \centering
    \includegraphics[width=\linewidth]{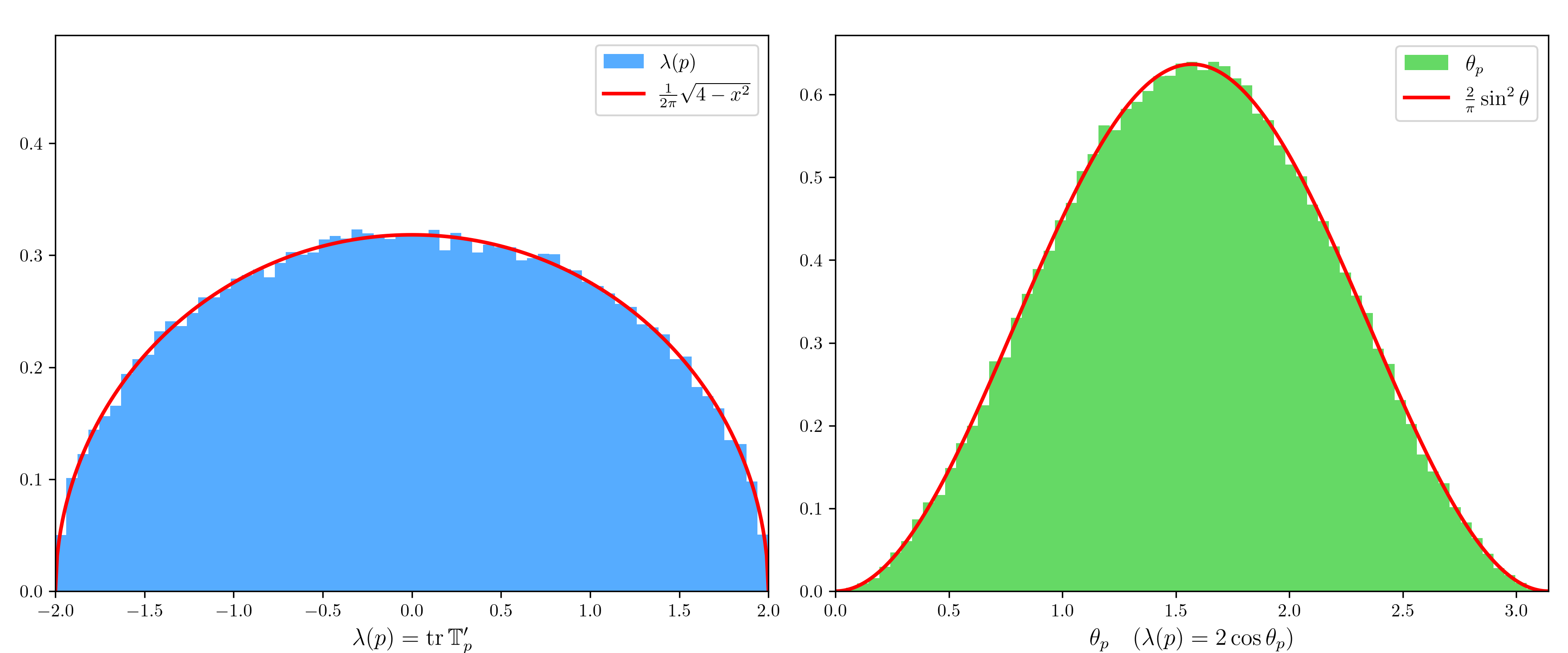}
    \caption{Sato--Tate plot for $\Q(\sqrt{5})$ and weight $(4,8)$. This has $148,931$ totally positive prime elements (giving here $297,862$ data points as we can consider both Galois embeddings) in $65$ bins.}
    \label{fig:SatoTatePlot_5_48}
\end{figure}

\medskip
{\em Acknowledgments.}
The second author has been supported by the National Research Foundation of Korea (NRF) grant funded by the Korea government (MSIP) (No. RS-2024-00415601 (G-BRL)). The third author is partially supported by JSPS Grant-in-Aid for Scientific Research (B) No.23K20785.

\bibliography{bibliography}

\begin{bibdiv}
\begin{biblist}

\bib{Art89}{article}{
      author={Arthur, James},
       title={The {$L^2$}-{L}efschetz numbers of {H}ecke operators},
        date={1989},
        ISSN={0020-9910,1432-1297},
     journal={Invent. Math.},
      volume={97},
      number={2},
       pages={257\ndash 290},
         url={https://doi.org/10.1007/BF01389042},
      review={\MR{1001841}},
}

\bib{Bla06}{incollection}{
      author={Blasius, Don},
       title={Hilbert modular forms and the {R}amanujan conjecture},
        date={2006},
   booktitle={Noncommutative geometry and number theory},
      series={Aspects Math.},
      volume={E37},
   publisher={Friedr. Vieweg, Wiesbaden},
       pages={35\ndash 56},
         url={https://doi.org/10.1007/978-3-8348-0352-8_2},
      review={\MR{2327298}},
}

\bib{SatoTate}{article}{
      author={Barnet-Lamb, Thomas},
      author={Gee, Toby},
      author={Geraghty, David},
       title={The {S}ato-{T}ate conjecture for {H}ilbert modular forms},
        date={2011},
        ISSN={0894-0347,1088-6834},
     journal={J. Amer. Math. Soc.},
      volume={24},
      number={2},
       pages={411\ndash 469},
         url={https://doi.org/10.1090/S0894-0347-2010-00689-3},
      review={\MR{2748398}},
}

\bib{Bou04}{book}{
      author={Bourbaki, Nicolas},
       title={Integration. {II}. {C}hapters 7--9},
      series={Elements of Mathematics (Berlin)},
   publisher={Springer-Verlag, Berlin},
        date={2004},
        ISBN={3-540-20585-3},
        note={Translated from the 1963 and 1969 French originals by Sterling K. Berberian},
      review={\MR{2098271}},
}

\bib{CD90}{article}{
      author={Clozel, Laurent},
      author={Delorme, Patrick},
       title={Le th{\'e}or{\`e}me de {Paley}-{Wiener} invariant pour les groupes de {Lie} r{\'e}ductifs. {II}. ({The} invariant {Paley}-{Wiener} theorem for reductive {Lie} groups. {II})},
    language={French},
        date={1990},
        ISSN={0012-9593},
     journal={Ann. Sci. {\'E}c. Norm. Sup{\'e}r. (4)},
      volume={23},
      number={2},
       pages={193\ndash 228},
         url={https://eudml.org/doc/82271},
}

\bib{Coh75}{article}{
      author={Cohen, Henri},
       title={Sums involving the values at negative integers of {$L$}-functions of quadratic characters},
        date={1975},
        ISSN={0025-5831,1432-1807},
     journal={Math. Ann.},
      volume={217},
      number={3},
       pages={271\ndash 285},
         url={https://doi.org/10.1007/BF01436180},
      review={\MR{382192}},
}

\bib{Dat93}{incollection}{
      author={Datskovsky, Boris~A.},
       title={A mean-value theorem for class numbers of quadratic extensions},
    language={English},
        date={1993},
   booktitle={A tribute to emil grosswald: number theory and related analysis},
   publisher={Providence, RI: American Mathematical Society},
       pages={179\ndash 242},
}

\bib{Eic55}{article}{
      author={Eichler, Martin},
       title={On the class of imaginary quadratic fields and the sums of divisors of natural numbers},
        date={1955},
        ISSN={0019-5839,2455-6475},
     journal={J. Indian Math. Soc. (N.S.)},
      volume={19},
       pages={153\ndash 180},
      review={\MR{80769}},
}

\bib{Eic57}{article}{
      author={Eichler, M.},
       title={Eine {V}erallgemeinerung der {A}belschen {I}ntegrale},
        date={1957},
        ISSN={0025-5874,1432-1823},
     journal={Math. Z.},
      volume={67},
       pages={267\ndash 298},
         url={https://doi.org/10.1007/BF01258863},
      review={\MR{89928}},
}

\bib{Gar90}{book}{
      author={Garrett, Paul~B.},
       title={Holomorphic {H}ilbert modular forms},
      series={The Wadsworth \& Brooks/Cole Mathematics Series},
   publisher={Wadsworth \& Brooks/Cole Advanced Books \& Software, Pacific Grove, CA},
        date={1990},
        ISBN={0-534-10344-8},
      review={\MR{1008244}},
}

\bib{Hij74}{article}{
      author={Hijikata, H.},
       title={Explicit formula of the traces of {Hecke} operators for {{\(\Gamma_0(N)\)}}},
    language={English},
        date={1974},
        ISSN={0025-5645},
     journal={J. Math. Soc. Japan},
      volume={26},
       pages={56\ndash 82},
}

\bib{Hur85}{article}{
      author={Hurwitz, Adolf},
       title={Ueber {R}elationen zwischen {C}lassenanzahlen bin\"arer quadratischer {F}ormen von negativer {D}eterminante},
        date={1885},
        ISSN={0025-5831,1432-1807},
     journal={Math. Ann.},
      volume={25},
      number={2},
       pages={157\ndash 196},
         url={https://doi.org/10.1007/BF01446402},
      review={\MR{1510301}},
}

\bib{HW18}{book}{
      author={Hoffmann, Werner},
      author={Wakatsuki, Satoshi},
       title={On the geometric side of the {Arthur} trace formula for the symplectic group of rank 2},
    language={English},
      series={Mem. Am. Math. Soc.},
   publisher={Providence, RI: American Mathematical Society (AMS)},
        date={2018},
      volume={1224},
        ISBN={978-1-4704-3102-0; 978-1-4704-4825-7},
}

\bib{HZ76}{article}{
      author={Hirzebruch, F.},
      author={Zagier, D.},
       title={Intersection numbers of curves on {H}ilbert modular surfaces and modular forms of {N}ebentypus},
        date={1976},
        ISSN={0020-9910,1432-1297},
     journal={Invent. Math.},
      volume={36},
       pages={57\ndash 113},
         url={https://doi.org/10.1007/BF01390005},
      review={\MR{453649}},
}

\bib{Ish74}{article}{
      author={Ishikawa, Hirofumi},
       title={On trace of {H}ecke operators for discontinuous groups operating on the product of the upper half planes},
        date={1974},
        ISSN={0040-8980},
     journal={J. Fac. Sci. Univ. Tokyo Sect. IA Math.},
      volume={21},
       pages={357\ndash 376},
      review={\MR{364105}},
}

\bib{Ish88}{article}{
      author={Ishikawa, Hirofumi},
       title={A table of the dimensions of the {H}ilbert modular type cusp forms over real quadratic fields},
        date={1988},
        ISSN={0386-2194},
     journal={Proc. Japan Acad. Ser. A Math. Sci.},
      volume={64},
      number={3},
       pages={84\ndash 87},
         url={http://projecteuclid.org/euclid.pja/1195513362},
      review={\MR{952813}},
}

\bib{JRW06}{incollection}{
      author={Jacobson, Michael~J.},
      author={Ramachandran, Shantha},
      author={Williams, Hugh~C.},
       title={Numerical results on class groups of imaginary quadratic fields},
    language={English},
        date={2006},
   booktitle={Algorithmic number theory. 7th international symposium, ants-vii, berlin, germany, july 23--28, 2006. proceedings.},
   publisher={Berlin: Springer},
       pages={87\ndash 101},
}

\bib{KL06}{book}{
      author={Knightly, Andrew},
      author={Li, Charles},
       title={Traces of {H}ecke operators},
      series={Mathematical Surveys and Monographs},
   publisher={American Mathematical Society, Providence, RI},
        date={2006},
      volume={133},
        ISBN={978-0-8218-3739-9; 0-8218-3739-7},
         url={https://doi.org/10.1090/surv/133},
      review={\MR{2273356}},
}

\bib{Kro60}{article}{
      author={Kronecker, L.},
       title={Ueber die {A}nzahl der verschiedenen {C}lassen quadratischer {F}ormen von negativer {D}eterminante},
        date={1860},
        ISSN={0075-4102,1435-5345},
     journal={J. Reine Angew. Math.},
      volume={57},
       pages={248\ndash 255},
         url={https://doi.org/10.1515/crll.1860.57.248},
      review={\MR{1579129}},
}

\bib{KTW22}{article}{
      author={Kim, Henry~H.},
      author={Tsuzuki, Masao},
      author={Wakatsuki, Satoshi},
       title={The {Shintani} double zeta functions},
    language={English},
        date={2022},
        ISSN={0933-7741},
     journal={Forum Math.},
      volume={34},
      number={2},
       pages={469\ndash 505},
}

\bib{Lan76}{book}{
      author={Lang, Serge},
       title={Introduction to modular forms},
      series={Grundlehren der Mathematischen Wissenschaften},
   publisher={Springer-Verlag, Berlin-New York},
        date={1976},
      volume={No. 222},
      review={\MR{429740}},
}

\bib{lmfdb:4.0.125.1}{misc}{
      author={{LMFDB Collaboration}, The},
       title={The {L}-functions and modular forms database, {H}ome page of number field \texttt{4.0.125.1}},
         how={\mbox{\url{https://www.lmfdb.org/NumberField/4.0.125.1}}},
        date={2026},
        note={[Online; accessed 19 June 2026]},
}

\bib{lmfdb:2.2.29.1-1.1-a-curve}{misc}{
      author={{LMFDB Collaboration}, The},
       title={The {L}-functions and modular forms database, {H}ome page of the {E}lliptic curve isogeny class 1.1-a over number field $\mathbb{Q}(\sqrt{29}$)},
         how={\mbox{\url{https://www.lmfdb.org/EllipticCurve/2.2.29.1/1.1/a/}}},
        date={2026},
        note={[Online; accessed 19 June 2026]},
}

\bib{lmfdb:2.2.29.1-1.1-a}{misc}{
      author={{LMFDB Collaboration}, The},
       title={The {L}-functions and modular forms database, {H}ome page of the {H}ilbert cusp form \texttt{2.2.29.1-1.1-a}},
         how={\mbox{\url{https://www.lmfdb.org/ModularForm/GL2/TotallyReal/2.2.29.1/holomorphic/2.2.29.1-1.1-a}}},
        date={2026},
        note={[Online; accessed 19 June 2026]},
}

\bib{SO94}{article}{
      author={Louboutin, St\'ephane},
      author={Okazaki, Ryotaro},
       title={Determination of all non-normal quartic {CM}-fields and of all non-abelian normal octic {CM}-fields with class number one},
        date={1994},
        ISSN={0065-1036,1730-6264},
     journal={Acta Arith.},
      volume={67},
      number={1},
       pages={47\ndash 62},
         url={https://doi.org/10.4064/aa-67-1-47-62},
      review={\MR{1292520}},
}

\bib{Mer14}{article}{
      author={Mertens, Michael~H.},
       title={Mock modular forms and class number relations},
        date={2014},
        ISSN={2522-0144,2197-9847},
     journal={Res. Math. Sci.},
      volume={1},
       pages={Art. 6, 16},
         url={https://doi.org/10.1186/2197-9847-1-6},
      review={\MR{3377995}},
}

\bib{Neu99}{book}{
      author={Neukirch, J\"urgen},
       title={Algebraic number theory},
      series={Grundlehren der mathematischen Wissenschaften [Fundamental Principles of Mathematical Sciences]},
   publisher={Springer-Verlag, Berlin},
        date={1999},
      volume={322},
        ISBN={3-540-65399-6},
         url={https://doi.org/10.1007/978-3-662-03983-0},
        note={Translated from the 1992 German original and with a note by Norbert Schappacher, With a foreword by G. Harder},
      review={\MR{1697859}},
}

\bib{Oka02}{article}{
      author={Okada, Kaoru},
       title={Hecke eigenvalues for real quadratic fields},
    language={English},
        date={2002},
        ISSN={1058-6458},
     journal={Exp. Math.},
      volume={11},
      number={3},
       pages={407\ndash 426},
         url={https://eudml.org/doc/53262},
}

\bib{PR94}{book}{
      author={Platonov, Vladimir},
      author={Rapinchuk, Andrei},
       title={Algebraic groups and number theory},
      series={Pure and Applied Mathematics},
   publisher={Academic Press, Inc., Boston, MA},
        date={1994},
      volume={139},
        ISBN={0-12-558180-7},
        note={Translated from the 1991 Russian original by Rachel Rowen},
      review={\MR{1278263}},
}

\bib{Pre68}{article}{
      author={Prestel, Alexander},
       title={Die elliptischen {F}ixpunkte der {H}ilbertschen {M}odulgruppen},
        date={1968},
        ISSN={0025-5831,1432-1807},
     journal={Math. Ann.},
      volume={177},
       pages={181\ndash 209},
         url={https://doi.org/10.1007/BF01350863},
      review={\MR{228439}},
}

\bib{Sai84}{article}{
      author={Saito, Hiroshi},
       title={On an operator {{\(U_{\chi}\)}} acting on the space of {Hilbert} cusp forms},
    language={English},
        date={1984},
        ISSN={0023-608X},
     journal={J. Math. Kyoto Univ.},
      volume={24},
       pages={285\ndash 303},
}

\bib{Sel56}{article}{
      author={Selberg, A.},
       title={Harmonic analysis and discontinuous groups in weakly symmetric {R}iemannian spaces with applications to {D}irichlet series},
        date={1956},
        ISSN={0019-5839,2455-6475},
     journal={J. Indian Math. Soc. (N.S.)},
      volume={20},
       pages={47\ndash 87},
      review={\MR{88511}},
}

\bib{Ser97}{article}{
      author={Serre, Jean-Pierre},
       title={R\'epartition asymptotique des valeurs propres de l'op\'erateur de {H}ecke {$T_p$}},
        date={1997},
        ISSN={0894-0347,1088-6834},
     journal={J. Amer. Math. Soc.},
      volume={10},
      number={1},
       pages={75\ndash 102},
         url={https://doi.org/10.1090/S0894-0347-97-00220-8},
      review={\MR{1396897}},
}

\bib{Shi63}{article}{
      author={Shimizu, Hideo},
       title={On discontinuous groups operating on the product of the upper half planes},
        date={1963},
        ISSN={0003-486X},
     journal={Ann. of Math. (2)},
      volume={77},
       pages={33\ndash 71},
         url={https://doi.org/10.2307/1970201},
      review={\MR{145106}},
}

\bib{Shi65}{article}{
      author={Shimizu, Hideo},
       title={On zeta functions of quaternion algebras},
        date={1965},
        ISSN={0003-486X},
     journal={Ann. of Math. (2)},
      volume={81},
       pages={166\ndash 193},
         url={https://doi.org/10.2307/1970389},
      review={\MR{171771}},
}

\bib{ST20}{article}{
      author={Sugiyama, Shingo},
      author={Tsuzuki, Masao},
       title={Optimal estimates for an average of {H}urwitz class numbers},
        date={2020},
        ISSN={1382-4090,1572-9303},
     journal={Ramanujan J.},
      volume={52},
      number={1},
       pages={91\ndash 104},
         url={https://doi.org/10.1007/s11139-019-00146-z},
      review={\MR{4083224}},
}

\bib{ST61}{book}{
      author={Shimura, Goro},
      author={Taniyama, Yutaka},
       title={Complex multiplication of abelian varieties and its applications to number theory},
      series={Publications of the Mathematical Society of Japan},
   publisher={Mathematical Society of Japan, Tokyo},
        date={1961},
      volume={6},
      review={\MR{125113}},
}

\bib{Str14}{article}{
      author={Streng, Marco},
       title={Computing {I}gusa class polynomials},
        date={2014},
        ISSN={0025-5718,1088-6842},
     journal={Math. Comp.},
      volume={83},
      number={285},
       pages={275\ndash 309},
         url={https://doi.org/10.1090/S0025-5718-2013-02712-3},
      review={\MR{3120590}},
}

\bib{Su16}{article}{
      author={Su, Ren-He},
       title={Eisenstein series in the {K}ohnen plus space for {H}ilbert modular forms},
        date={2016},
        ISSN={1793-0421,1793-7310},
     journal={Int. J. Number Theory},
      volume={12},
      number={3},
       pages={691\ndash 723},
         url={https://doi.org/10.1142/S1793042116500469},
      review={\MR{3477416}},
}

\bib{Tak86}{article}{
      author={Takase, Koichi},
       title={On the trace formula of the {H}ecke operators and the special values of the second {$L$}-functions attached to the {H}ilbert modular forms},
        date={1986},
        ISSN={0025-2611,1432-1785},
     journal={Manuscripta Math.},
      volume={55},
      number={2},
       pages={137\ndash 170},
         url={https://doi.org/10.1007/BF01168682},
      review={\MR{833241}},
}

\bib{PARI2}{manual}{
      author={{The PARI~Group}},
       title={{PARI/GP version \texttt{2.17.3}}},
     address={Univ. Bordeaux},
        date={2026},
        note={available from \url{http://pari.math.u-bordeaux.fr/}},
}

\bib{Zag75}{article}{
      author={Zagier, Don},
       title={Nombres de classes et formes modulaires de poids {$3/2$}},
        date={1975},
        ISSN={0151-0509},
     journal={C. R. Acad. Sci. Paris S\'er. A-B},
      volume={281},
      number={21},
       pages={Ai, A883\ndash A886},
      review={\MR{429750}},
}

\bib{Zag77}{incollection}{
      author={Zagier, D.},
       title={Correction to: ``{T}he {E}ichler-{S}elberg trace formula on {${\rm SL}\sb{2}({\bf Z})$}'' ({\it {i}ntroduction to modular forms}, {A}ppendix, pp. 44--54, {S}pringer, {B}erlin, 1976) by {S}. {L}ang},
        date={1977},
   booktitle={Modular functions of one variable, {VI} ({P}roc. {S}econd {I}nternat. {C}onf., {U}niv. {B}onn, {B}onn, 1976)},
      series={Lecture Notes in Math.},
      volume={Vol. 627},
   publisher={Springer, Berlin-New York},
       pages={171\ndash 173},
      review={\MR{480354}},
}

\end{biblist}
\end{bibdiv}
\bibliographystyle{amsplain}
\end{document}